%% file: Arxiv-final.tex
\documentclass[12pt,reqno,a4paper]{amsart}
\usepackage{extsizes}
\usepackage{blindtext}

\usepackage{amssymb}
\usepackage{amsmath}
\usepackage{amsthm}
\usepackage{amsbsy}
\usepackage{bm}
\usepackage{hyperref}
\date{\today}
\usepackage{cite}
\usepackage{array}
\usepackage{graphicx}

\input{mypreamble}
\title[Fluctuation and entropy]{  Fluctuation and Entropy in Spectrally Constrained random fields}

\author{Kartick Adhikari}
\address{Department of Mathematics, Indian Institute of Science Education and Research, Bhopal, 462066, India}

\email{kartickmath [at] gmail.com}

\author{Subhroshekhar Ghosh}
\address{Department of Mathematics, National University of Singapore, 10 Lower Kent Ridge Road,
	Singapore 119076}

\email{ subhrowork [at] gmail.com}

\author{Joel L. Lebowitz}
\address{Departments of Mathematics and Physics,
	Rutgers University, Piscataway, New Jersey 08854, USA}

\email{lebowitz [at] math.rutgers.edu}

\date{}

\begin{document}
	\maketitle
\vspace{-15pt}
\begin{abstract}
We investigate the statistical properties of translation invariant random fields (including point processes) on Euclidean spaces (or lattices) under constraints on their spectrum or structure function. An important  class of models that  motivate our study are hyperuniform and stealthy hyperuniform systems, which are characterised by the vanishing of the structure function at the origin (resp., vanishing in a neighbourhood of the origin). We show that many key features of two classical statistical mechanical measures of randomness - namely, fluctuations and entropy, are governed only by some particular local aspects of their structure function.  We  obtain exponents for the fluctuations of the local mass in domains of growing size, and show that spatial geometric considerations play an important role - both  the shape of the domain and the mode of spectral decay. In doing so, we unveil intriguing oscillatory behaviour of spatial correlations of local masses in adjacent box domains. We describe very general conditions under which we show that the field of local masses exhibit Gaussian asymptotics, with an explicitly described limit. We further demonstrate that stealthy hyperuniform systems with joint densities exhibit degeneracy in their asymptotic entropy per site. 
In fact, our analysis shows that entropic degeneracy sets in under much milder conditions than stealthiness, as soon as the structure function fails to be logarithmically integrable.
\end{abstract}

\section{Introduction and Main Results}

\subsection{Setup and notations}
Random fields with spectral constraints have attracted considerable attention in recent years, both in the mathematics and the statistical and condensed matter physics communities. An important motivation, from the  point of view of applications, comes from the investigation of hyperuniform point fields, whose structure function vanishes at the origin. A sub-topic of particular interest recently in condensed matter physics is the study of stealthy hyperuniform systems,  whose structure function vanishes in a region containing the origin.

In this work, we will concern ourselves with random fields that live on Euclidean spaces. For ease of enunciating and analysing the statements of results, we will focus on the situation where the random field $(X_i)_{i \in \Z^d}$ actually lives on a Euclidean lattice, although most of our results have straightforward generalisations to the case of  random fields indexed by a continuum Euclidean space. Restricting ourselves to random fields that are invariant under translations in $\Z^d$ implies that their spatial correlation function depends only on their spatial separation. The spectrum, also known as the  structure function, of the process can then be defined as the Fourier transform of this pair correlation function. We are interested in investigating the properties of  the random field $X$ under the assumption that the spectrum satisfies certain constraints. 

We now introduce the exact definitions and concrete notations. Let $X=(X_i)_{i\in \Z^d}$ be a real valued mean zero and variance one translation invariant random field with covariance
\begin{align*}
\E[X_iX_j]=K(i-j), \mbox{  for $i,j\in \Z^d$,}
\end{align*}
where $K$ is defined on $\Z^d$ and  $\{K(j)\}_{j\in \Z^d}\in \ell_1(\Z^d)$. Note that $K(0)=1$ as the variance is one.   Let $S$  be the Fourier transform of $K$.  Then $S$ is defined on $[-\pi, \pi]^d$ by
\begin{align}\label{eqn:S}
S(\theta)=\sum_{j\in \Z^d}e^{i\theta \cdot j}K(j), \mbox{ where $\theta\in [-\pi,\pi]^d$}.
\end{align}
Here $\theta\cdot j=\sum_{k=1}^d\theta_kj_k$, the usual dot product in the $d$-dimensional Euclidean space.  This function $S$ is known as the {\it diffraction spectrum or structure function} of the field $X$. Observe that $S$ is symmetric, i.e., $S(-\theta)=S(\theta)$, and  Bochner's theorem \cite[Theorem 4.18]{folland} implies that $S$ is non-negative.

For any $\Lambda \subset \R^d$, we denote by $Q_{\Lambda}$   the  charge (or the \textit{local mass}) carried by the random variables located in the region $\Lambda$, i.e.,
\begin{align*}
Q_{\Lambda}(X)=\sum_{i\in \Lambda}X_i.
\end{align*}
It is clear that $\E[Q_{\Lambda}(X)]=0$. We study the fluctuations of $Q_{\mathcal C_L}(X)$ and $Q_{\mathcal B_L}(X)$, where, for $x=(x_1,\ldots, x_d)$, 
\begin{align*}
\mathcal C_L&=\{x\in \R^d: |x_1|, \ldots, |x_d|\le L\} \mbox{ and } \mathcal B_L=\{x\in \R^d: \sqrt{ x_1^2+\cdots+ x_d^2}\le L\}.
\end{align*}
More precisely, the asymptotic values of $\Var(Q_{\CC_L}(X))$ and  $\Var(Q_{\BB_L}(X))$ are calculated, as $L\to \infty$, where $\Var(Y)$ denotes the variance of $Y$. This is done  under various conditions on the kernel function $K$, equivalently, on the structure function $S$.  
We also show that under the appropriate condition on $S$   
 \begin{align*}
 \frac{Q_{\BB_L}(X)-\E[Q_{\BB_L}(X)]}{\sqrt{\Var(Q_{\BB_L}(X))}}\to \mathcal N(0,1), \mbox{ as $L\to \infty$},
\end{align*}
where $\mathcal N(0,1)$ denotes the standard normal distribution.

Finally we study the asymptotic behaviour of the entropy of $X|_{\Lambda}:=\{X_i\; :\;i\in \Lambda\}$, as the size of the region $\Lambda$ increases, in the setting where $X$ has joint densities.

\subsection{Fluctuations for constrained systems}
A key measure of randomness that we will focus on in this paper is fluctuations of the local field of masses. This is motivated by the study of point processes, but as we shall see, it can be considered also in the context of more general random fields, and with important structural consequences. While the exact definitions will be given  subsequently it suffices for our introductory discussion to have in mind that the local field of masses of a random field (indexed by a Euclidean lattice), pertaining to a given domain $D$, is the sum of the field values for indices that belong to $D$. Clearly, if the field is 0-1 valued (in other words, a point process), then this quantity reduces to the total number of particles in the domain $D$. In this work, without loss of generality we focus on random field with mean zero, since subtracting  the mean has no effect on the  variance. The study of the statistical fluctuations of masses has a long history, see e.g., \cite{beck87}, \cite{beck1987chen} and the references therein, and e.g. \cite{GhL-survey} for a more recent overview.

The starting point for examining the connection between spectral decay and reduced fluctuations of the local mass is the fact, alluded to earlier, that hyperuniformity is equivalent to vanishing of the structure function at the origin. For the Gaussian Unitary Ensemble (GUE) in 1D or the Ginibre ensemble in 2D (equivalently, the one component 2D Coulomb gas at the  inverse temperature $\beta=2$), this amounts to a linear decay at the origin, and fluctuation of local mass that is logarithmic (for GUE) and of the order of the perimeter of the domain (for Ginibre). In both cases, asymptotic normality of the fluctuations is also well-understood.

In this paper we investigate the precise quantitative nature of the relationship  between the mode of spectral decay on one hand,  and the asymptotic statistical distribution of fluctuation on the other. We demonstrate that the correspondence between spectral decay and the statistical behaviour of the fluctuations  is very general, and is largely independent of other properties of the stochastic process. 
\vspace{-3pt}
\subsubsection{ \textbf{Fluctuations in Balls } }
Our first result provides a functional of the structure function that, on one hand, is asymptotically comparable to  the variance of the mass in a ball, and, on the other hand, is amenable to simple analytical determination of the growth exponent. This result holds in great generality, in particular, both for random fields indexed by $\Z^d$ and by $\R^d$.  We state the precise result below.

Let $\BB_L$ denote the ball of radius $L$ in $\R^d$, i.e.,
$
\BB_L=\{x\in \R^d \; : \; \|x\|\le L \}. 
$ Furthermore, let $\S^{d-1}$ denote the $(d-1)$-dimensional unit sphere in $\R^d$ and $d V_{\S^{d-1}}$ denote the $(d-1)$-dimensional spherical measure. Let $\mathbb{T}^d=[-\pi,\pi]^d$.
We call a function $f: \mathbb{T}^d \;(\mathrm{or }\; \R^d) \mapsto \R$  \textit{regular at the origin} if there exists an \textit{enveloping function} $A : \mathbb{T}^d \;(\mathrm{or }\; \R^d) \mapsto \R$ such that, on some neighbourhood $B(0;\eps)$ of the origin, we have $cA(\xi) \le f(\xi) \le C A(\xi)$ for some constants $c,C>0$, and the enveloping function $A$ is such that $\|\xi\|^{-2} \l({\int_{\S^{d-1}}A(\|\xi\|\omega)dV_{\S^{d-1}}(\omega)}\r)$ is monotone for $0\le \|\xi\|\le \epsilon$. In general, the function $f$ can possibly have highly fluctuating behaviour, it might be  difficult to work with such functions. On the other hand the enveloping function has nicer properties in the context of the estimate that we need, which helps in the calculations. 

For example,  $f$ might be highly fluctuating near the origin but $f(\xi)/\|\xi\|$ might be bounded by two positive constants. In that case $A(\xi)=\|\xi\|$, which implies that $\|\xi\|^{-2} \l({\int_{\S^{d-1}}A(\|\xi\|\omega)dV_{\S^{d-1}}(\omega)}\r)=\mbox{const.}/ \|\xi\|$ which is clearly monotone.

Furthermore, we recall the notation $f(L)=\Theta (g(L))$ as $L\to \infty$ to mean that  there exist constants  $c_1,c_2, L_0>0$ such that 
\[
c_1 g(L)\le f(L) \le c_2 g(L), \mbox{ for all $L>L_0$}.
\]
More generally, we write $f(x)=\Theta(g(x))$ as $x\to x_0$ if and only if there exist constants $c_1,c_2,\delta>0$ such that 
\[
c_1g(x)\le f(x)\le c_2g(x), \mbox{ for all $\|x-x_0\|\le \delta$.}
\]

We establish that
%
\begin{theorem}\label{thm:varianceball}
   Suppose the structure function $S$ of a random field on $\Z^d$ is regular at the origin and $S$ is bounded in $[-\pi,\pi]^d$.
	Then, as $L\to \infty$,
	\begin{align}\label{eq:varballeq} 
	\Var(Q_{\BB_{L}}(X))=\Theta\l( L^{2d}\int_{\|\xi\|\le {c}/{L}}S(\xi)d\xi + L^{d-1}\int_{\|\xi\|> {c}/{L}}\frac{S(\xi)}{\|\xi\|^{d+1}}d\xi\r). 
	\end{align}
 
\end{theorem}


The regularity assumption in Theorem \ref{thm:varianceball} holds true in great generality. It is trivially true in  situations where the structure function $S$ is itself monotone in $\|\xi\|$  near the origin, which is the case for the Poisson process and the standard examples of hyperuniform systems. More generally, by a Taylor expansion of $S$ near the origin, such regularity will be true as soon as $S$  has bounded derivatives near the origin.  

We observe that the asymptotic formula in Theorem \ref{thm:varianceball} provides a very general and complete description  for the fluctuation exponent in a ball that is valid in any dimension, and covers all  settings where this exponent is believed to be understood. Moreover, the right hand side of \eqref{eq:varballeq} lends itself to relatively simple estimation of its order in $L$ as soon as the behaviour of the structure function near the origin is known.   E.g., for Poissonian  systems, the structure function does not vanish near the origin (in fact, it is a positive constant), and as such, the first term in \eqref{eq:varballeq} provides the dominant contribution, resulting in fluctuations of the order of the volume. On the other hand, for well-known hyperuniform systems like the GUE in 1D (where $S$ is linear near the origin) and the Ginibre ensemble in 2D (where $S$ behaves like a quadratic near the origin), the second terms in \eqref{eq:varballeq} makes the dominant contribution, and the variance grows like $\log L$ and $L$ respectively.


\subsection{The geometry of spectral decay and its consequences}
The most intensively studied setting for hyperuniform systems is the 1D case, which includes the famous example of the GUE process from random matrix theory. However, in 1D, the variety of ways in which the  spectrum can vanish at the origin is relatively limited. In higher dimensions, the mode of decay of the  spectrum becomes important, and the physical manifestations of the various decay modes is a significant question.


\subsubsection{\textbf{Fluctuations under spectral decay}}
	
Let $n_1,\ldots, n_d\in \N\cup \{0\}$ and $\n=(n_1,\ldots, n_d)$. Denote
$$
\CC^{(\n)}_{L}=[n_1L,(n_1+1)L]\times \cdots \times [n_dL, (n_d+1)L].
$$
Note that $\CC^{(\n)}_{L}$ is a cube/box in $\R^d$ with side length $L$. In particular $\CC_L^{(0)}=[0,L]^d$. We have the following result.

\begin{theorem}\label{lem:covarianced}
	 Suppose  the structure function $S$ of a random field on $\Z^d$ satisfies $\sigma^2_d:=2^d\int_{-\pi}^{\pi}\cdots \int_{-\pi}^{\pi}\frac{S(x)}{x_1^2\cdots x_d^2}dx_1\cdots dx_d<\infty$. Then
	$$
	\lim_{L\to\infty}\cov(Q_{\CC^{(0)}_{L}}(X),Q_{\CC^{(\n)}_{L}}(X))=\l\{\begin{array}{ll}
	(-1)^j\frac{\sigma_d^2}{2^j} &\mbox{ if } \mbox{dim}(\CC_{L}^{(0)}\cap \CC_{L}^{(\n)})=d-j,\\
	\\
	0 & \mbox{ if } \CC_{L}^{(0)}\cap \CC_{L}^{(\n)}=\emptyset.
	\end{array}\r.
	$$
In particular, $\Var(Q_{\CC^{(0)}_{L}}(X))\to \sigma^2_d$ as $L\to \infty$.
\end{theorem}

   Here dim$(\CC_{L}^{(0)}\cap \CC_{L}^{(\n)})=0$ means the two cubes/boxes have only a common vertex point. 
   

In particular for $d=1$ we have the following corollary. The importance of the 1D case is that, besides being the most studied setting, it is also the situation where we can examine the effect of stealth or other spectral constraints without having to consider the intricacies of geometric effects.

   \begin{corollary}\label{lem:variance}
   	Suppose the structure function $S$ of a random field on $\Z$ satisfies  $\sigma^2:=2\int_{-\pi}^{\pi}\frac{S(x)}{x^2}dx<\infty$. Then
   	\begin{align*}
   		\Var(Q_{[0,L]}(X))\to \sigma^2 \;\mbox{ and }\;  \cov(Q_{[0,L]}(X),Q_{{[L,2L]}}(X))\to -\frac{\sigma^2}{2}, \;\mbox{ as $L\to \infty$},
   	\end{align*}
   	where    $\cov(X,Y)$ denotes the covariance of $X$ and $Y$. Moreover, for $k\ge 2$,
   	\begin{align*}
   		\cov(Q_{[0,L]}(X),Q_{[kL,(k+1)L]}(X))\to 0, \; \mbox{ as $L\to \infty$. }
   	\end{align*}
   \end{corollary}
   
   
   Note that if $S$ is bounded and vanishing in a neighbourhood of the origin then 
   $
   \int \frac{S(x)}{x^2}dx<\infty.
   $
   In this case the covariance converges to the half of the variance with negative sign when two intervals are adjacent to each other. Moreover, the covariance converges to $0$ when two intervals are not adjacent.
 
 The Riemann Lebesgue lemma plays a crucial role in proving these results. The assumption $\sigma_d^2<\infty$  is required to apply the Riemann Lebesgue lemma. In the next subsection we show that if $\sigma_d^2$ is not finite then the variances and covariances do not have finite limits as $L\to \infty$. We calculate the asymptotic behaviour of the variances and covariances under suitable conditions on $S$.
 
Note that Theorem \ref{lem:covarianced}  implies that if $\sigma_d$ is finite then $\Var(Q_{ \CC_{L}}(X))$ remains bounded in $L$,   whereas Theorem \ref{thm:varianceball} implies that the order of $\Var(Q_{\BB_{L}}(X))$ is at least $L^{d-1}$ (also c.f. \cite{beck87}). A heuristic explanation for this phenomenon is given in Section \ref{sec:explanation}.
 
 It is worth mentioning here that the similar results for Ginibre ensembles were established in \cite{lebowitz1983charge}. The similar problems for the zeros of Gaussian entire functions were consider in \cite{buckely-sodin}.
 
\subsubsection{\textbf{Growth of variances under  spectral decay}} We  have the following result. We use $x=(x_1,\ldots, x_d)\in \R^d$ and $\CC_L=\CC_{L}^{(0)}$.
\begin{theorem}\label{lem:variancenonzerodge2}
	Let  $\alpha_1,\ldots, \alpha_d\in [0,1]$.  Suppose the structure function $S$ of a random field on $\Z^d$ satisfies  
\begin{enumerate}
	\item[(C1)] For all $k=1,\ldots, d$ and all $1\le i_1<\cdots<i_k\le d$, ${S(x)}=\Theta(|x_{i_1}|^{\alpha_{i_1}}\ldots |x_{i_k}|^{\alpha_{i_k}})$
	whenever $x_{i_1},\ldots,x_{i_k}\to 0$ and the other co-ordinates are away from the axes.
	
	\item[(C2)] For $\delta>0$, $\int_{\delta}^{\pi}\cdots \int_{\delta}^{\pi}\frac{S(x)}{x_1^2\ldots x_d^2}dx_1\cdots dx_d<\infty$.
\end{enumerate}
Then we have
	\begin{align*}
	\Var(Q_{\CC_{L}}(X))=\Theta((\log L)^{\tau_d}L^{d-m_{d}}), \mbox{ as $L\to \infty$},
	\end{align*}
	where $\tau_d=|\{k\in \{ 1,\ldots, d\}\; : \; \alpha_k=1\}|$ and $m_{d}=\sum_{k=1}^{d}\alpha_k$. 
\end{theorem}

\begin{remark}
Condition  (C1) can be written in the following way:   for $1 \le i_1, \ldots, i_k \le d $ and $\delta>0$, let  $A_{i_1.\ldots,i_k}^{\delta}$ be the set $\{ x=(x_1,\ldots,x_d) : |x_i| < \delta \;\mathrm{ iff }\; i \in  \{i_1,\ldots,i_k\} \}$. Then $S(x)=\Theta(|x_{i_1}|^{\alpha_{i_1}}\cdots |x_{i_k}|^{\alpha_{i_k}})$ whenever $x \in A_{i_1.\ldots,i_k}^{\delta}$. 

Note that if $S$ is bounded on $[-\pi,\pi]^d$ then (C2) holds. Also if $\alpha_i>1$ for all $i=1,\ldots,k$ then $\frac{S(x)}{x_1^{2}\cdots x_d^{2}}$ is integrable . In this case we get the variance from Theorem \ref{lem:covarianced}. So the variances are bounded when $\alpha_1,\ldots, \alpha_d> 1$. 

Observe that the implying constants for the variance may depend on the dimension $d$.
\end{remark}

It would be of interest to understand the various modes of spectral decay, especially in the context of hyperuniform behaviour of the stochastic process in  physical space. Hyperuniformity, or equivalently the vanishing of the structure function at the origin, implies that the correlation function $K(\underline{i})$ (which is a function of one variable because of translation invariance) sums to 0 over $\underline{i} \in \Z^d$. In particular, this implies that some of the correlations must be negative (since $K(0)$ is a variance and therefore necessarily positive), which explains the natural connections between hyperuniform systems and negatively associated processes. Decay of the structure function ``along the axes'' amounts to saying that $K(\underline{i})$ sums to 0 even when we sum $\underline{i}$ over any one co-ordinate, keeping the values of the other co-ordinates of $\underline{i}$ fixed. It is a somewhat stronger notion of hyperuniformity than mere vanishing of the structure function at the origin, and it covers examples as simple as  products of statistically independent 1D hyperuniform systems along each co-ordinate direction.  More generally, by the Bochner-Khinchine correspondence, we can have Gaussian stochastic processes with any given functional form for the decay of the structure function, as long as we remain within the realm of non-negative spectral measures.

The following result can be seen as a corollary of Theorem  \ref{lem:variancenonzerodge2}.  For $x=(x_1,\ldots, x_d)\in \R^d$, the $L^2$-norm is defined by
\[
\|x\|_2=(|x_1|^2+\cdots+|x_d|^2)^{\frac{1}{2}}.
\] 
\begin{corollary}\label{prop:d=2}
	Let   $0\le \alpha \le 1$. Suppose the structure function $S=\Theta(\|x\|_2^\alpha)$ as $\|x\|_2\to 0$, and $\int_{\|x\|_2>\delta} \frac{S(x)}{x_1^2\ldots x_d^2}dx<\infty$ for $\delta>0$. Then
	\begin{align*}
	\Var(Q_{\CC_{L}}(X))=\l\{ \begin{array}{ll}
	\Theta((L^{d-\alpha})) & \mbox{if $0\le \alpha<1$},\\
	&\\
	\Theta((L^{d-1}\log L)) & \mbox{if $\alpha=1$}.
	\end{array}
	\r.
	\end{align*}
\end{corollary}

Roughly speaking, we get Corollary \ref{prop:d=2} by putting $\alpha_1=\alpha$ and $\alpha_2=\cdots=\alpha_d=0$ in Theorem \ref{lem:variancenonzerodge2}.  


\subsubsection{\textbf{The effect of domain shape}}
An important consequence of our investigations is how the fluctuation exponent depends on shape of the growing domain (e.g., a ball vis-a-vis a cube). In fact, depending on the domain shape, even under relatively mild decay of the structure function, the fluctuations of the local mass can be bounded as the domain size grows to infinity. 

\subsubsection{\textbf{The anomaly of oscillating correlations}}
As seen in Theorem \ref{lem:covarianced} the spatial correlations of the \textit{field of local masses} (i.e., the masses in adjacent growing  cubes of similar sizes) exhibits a remarkable oscillating behaviour. To understand this phenomenon, we recall that hyperuniformity is often associated with repulsive interaction (or negatively correlated systems). Naturally, we expect this to be reflected in the spatial statistics of the field of local masses. E.g., for Coulomb systems in 3D, the leading order interactions  among local masses in adjacent cubes that meet in a face have been shown to be negative (c.f. \cite{lebowitz1983charge}, \cite{buckely-sodin}). This is in tune with the heuristic connection between hyperuniformity and negative dependence, and is similar in flavour to the $j=1$ case in Theorem \ref{lem:covarianced}.  However, Theorem \ref{lem:covarianced} goes further and unveils a more elaborate correlation landscape, depending on finer adjacency geometry of neighbouring domains for strongly hyperuniform systems. In fact, even the sign of the correlation can be positive or negative, depending on the dimension of the surface where two neighbouring domains intersect. 

In Section \ref{sec:explanation} , we explain this seemingly physically anomalous behaviour of the fluctuations (and the correlations of the local field of masses) from a microscopic statistical mechanical point of view, by showing that not only are these differential growth exponents and  oscillating signs of correlations consistent with each other, but also are necessary from a statistical physics perspective, and correspond naturally with the consideration of effects like Debye screening.

\subsection{Central limit theorem.} The next result show that $Q_{\BB_L}(X)$ is asymptotically normal, as $L\to \infty$, under appropriate conditions on the truncated correlation functions of $X$.
The $k$-th truncated correlation function $\rho_k^T$ of $X$ is defined by
\begin{align*}
\rho_k^T(i_1,\ldots, i_k)=\sum_{\pi\in \mathcal P(k)}(|\pi|-1)!(-1)^{|\pi|-1}\prod_{B\in \pi}\rho_B[i_1,\ldots, i_k],
\end{align*}
where $\mathcal P(k)$ denotes the set of all partitions of $\{1,\ldots, k\}$,  $|\pi|$ is the number of parts in the partition, $\rho_B[i_1,\ldots, i_k]=\rho_{|B|}(i_j\; :\; i_j\in B)$. 
The $k$-th intensity function $\rho_k$ of $X$ with respect to counting measure on $\Z^{kd}$ is given by
\begin{align*}
\rho_k(i_1,\ldots,i_k)=\E[X_{i_1}\cdots X_{i_k}], \mbox{ for } i_1,\ldots, i_k\in \Z^d.
\end{align*}
We elaborate further details of correlation and truncated correlation functions in Sections \ref{sec:cor} and \ref{sec:truncated}.

\begin{theorem}\label{thm:clt} Let $X=(X_i)_{i\in \Z^d}$ be a random field with structure function $S$ as in Theorem \ref{thm:varianceball}. Let $\rho^T(i_1,\ldots,i_k)$, for $i_1,\dots,i_k\in \Z^d$, be the truncated correlation functions of  $X$. Suppose  for each $k$ we have
\begin{align}\label{eqn:cltcondition}
	\sup_{i_1}\sum_{i_2,\ldots,i_n\in Z^d}\rho^T(i_1,\ldots,i_k)<\infty.
\end{align}
Let $\mathcal N(0,1)$ denote the standard normal distribution. Then 
\begin{align*}
	\frac{Q_{\BB_L}(X)-\E[Q_{\BB_L}(X)]}{\sqrt{\Var(Q_{\BB_L}(X))}}\to \mathcal N(0,1), \mbox{ as $L\to \infty.$}
\end{align*}
\end{theorem}

Theorem \ref{thm:clt} in general, and  condition \eqref{eqn:cltcondition} in particular,  connects to the classical theory of particle number fluctuations, as developed in  \cite{MaY}, \cite{jancovici}, \cite{OL} and references therein. It is explained in \cite[P. 446]{MaY} that the condition \eqref{eqn:cltcondition}  holds for a class of one-dimensional
Coulomb systems (these systems have exponential clustering), \cite{edwards1962exact}, \cite{aizenman1980structure}. The notion
of clustering and its relation to the truncated $k$-point function in rigorously introduced in \cite{nazarov2012correlation}.


\subsection{Entropy and entropic degeneracy for constrained systems }
A key parameter of randomness, or the lack thereof, is that of entropy per unit volume (in other words, entropy per site). We can envisage this as the entropy per unit volume of the field restricted to a finite domain of space, considered in the limit as the domain size grows to cover all space. For a system to be deficient in randomness, one measure would be  its entropy per site to be degenerate in some appropriate sense. 

One angle from which to look at maximal rigidity for stealthy random fields would be to consider it from the perspective of the tail sigma field. In one dimension, the notion of maximal rigidity under spatial conditioning can be demonstrated to be equivalent to the fullness of the two-sided spatial tail sigma field. It may be mentioned here that the well-known Rokhlin-Sinai Theorem  \cite[p. 322]{glasner} connects  the spatial tail sigma field with the entropy per site.
However, an extremely important caveat to the Rokhlin-Sinai Theorem is that it demands the much stronger assumption of fullness of one-sided spatial tail sigma field (and not two-sided, as we have in the models of our interest in 1D). In fact, the one-sidedness of the tail sigma field is known to be necessary for  the Rokhlin-Sinai Theorem,  and the lack of this particular characteristic makes an approach via Rokhlin-Sinai not tenable in our setting. In particular, it compels us to undertake a direct investigation of the entropy per site of stealthy (and indeed, other spectrally constrained) stochastic systems, invoking results connected to disparate areas of classical analysis and probability theory.

We will investigate the question of entropic degeneracy in the context of random fields having joint densities, that is, for any finite set $\Lambda \subset \Z^d$, the finite collection of random variables $(X_{\mathbf{i}})_{\mathbf{i} \in \Lambda}$ has a joint density on $\R^{|\Lambda|}$. Next, we revisit the concept of entropy for such random variables.

Let $X$ be a continuous random variable with probability density function $f$. Then the entropy of $X$ is defined by 
\begin{align*}
	h(X)=\E[-\log f(X)]=-\int f(x)\log f(x)dx.
\end{align*}
   Let $\Lambda\subset \R^d$ and $\Lambda'=\Lambda \cap \Z^d$. The cardinality of $\Lambda'$ is denoted by $|\Lambda'|$. Define $X|_{\Lambda}:=\{X_i\; : \; i\in \Lambda'\}$. The entropy of $X|_{\Lambda}$ is defined by
\begin{align*}
	h(X|_{\Lambda})=-\int_{\R^{|\Lambda'|}} f_\Lambda(x)\log f_\Lambda(x)dx,
\end{align*} 
where $f_\Lambda$ denotes the joint density function of $X|_{\Lambda}$.  Denote $\Lambda_L:=\{Lx\; : \; x \in \Lambda\}$ for $L>0$. 
In the next result we assume that the boundary of $\Lambda$ is $2$-smooth with positive Gaussian curvature,  the Gaussian curvature of a boundary at a  point is the product of the principal curvatures. A boundary is said to be 2-smooth boundary if it is locally given by a level of a $2$-smooth function.

\begin{theorem}\label{lem:entropycontinuous}
	Let $X=(X_i)_{i\in \Z^d}$ be a real valued mean zero and variance one translation invariant random field with joint densities on finite domains. Let $\Lambda$ be a bounded connected domain in $\R^d$ with $2$-smooth boundary $\partial \Lambda$  each connected component of which has positive Gaussian curvature. Suppose $|\Lambda_L\cap \Z^d|=O(L^d)$ as $L \to \infty$.
	Then, if the structure function $S$ of $X$   vanishes near the origin and satisfies the Sobolev condition, $\sum_{j\in \Z^d}|\hat S(j)|, \sum_{j\in \Z^d}||j|\hat S(j)|^2<\infty$, 
	\begin{align*}
	\mathcal{H}(X):= \lim_{L \to \infty} \frac{h(X|_{\Lambda_L})}{|\Lambda_L\cap \Z^d|} =  -\infty.
	\end{align*}
	
\end{theorem}

\begin{remark} \label{rem:entropy}
More generally, the proof of Theorem \ref{lem:entropycontinuous} demonstrates that 
entropic degeneracy is much more general than the vanishing of the structure function near the origin, and sets in as soon as the structure function fails to be logarithmically integrable. 
\end{remark}

%
%


\section{Statistical physics connections}
\subsection{Hyperuniformity}
Hyperuniform (also known as super homogeneous) processes are statistical mechanical systems that exhibit a higher level of uniformity than processes that can be considered to be purely random. In the domain of random point fields, the role of ``pure randomness'' is played by the Poisson process, which entails that points in disjoint spatial domains are statistically independent of each other. To the contrary, hyperuniform processes exhibit strong spatial correlation, which in particular acts to provide a measure of regularity that is noticably higher than the Poisson process. Thus, hyperuniform systems lie somewhere in between purely random and purely crystalline states of matter, which explains the interest in them from the perspective of condensed matter physics. A large gamut of literature has emerged in recent years that address the investigation of such systems, see e.g. \cite{AiMa},\cite{MaY}, \cite{GoLSp}, \cite{GhL-CMP}, \cite{ToSt}, \cite{MaST}, \cite{JiLHMCT}, \cite{JiT}, \cite{FlTS}, \cite{BuFN}, \cite{HaMS}, \cite{DeSMRPRJTB},\cite{HeL}, \cite{HeCL}, to provide a partial list. For an overview of this fairly large body of literature, we refer the interested reader to \cite{To-1}, \cite{GhL-survey}, and the references therein.

An important aspect of hyperuniform point processes is the  fluctuations of the particle count in a large domain of space. For the Poisson process, this  scales like the volume, while for a hyperuniform system it grows  slower than the volume, e.g., it may scale like the surface area (or even slower) of the domain. Hyperuniform processes cover a wide class of examples of natural statistical mechanical systems, principal among them being (one component) Coulomb systems, determinantal processes, and their derivatives. Hyperuniform processes arise naturally in the investigation of spectrally constrained random fields. For translation invariant processes, hyperuniformity can be shown to be equivalent to the vanishing of the spectrum at the origin of the frequency domain, and thus spectral considerations are naturally motivated in the study of such systems (see, e.g., \cite{GhL-CMP}, \cite{GhL-survey}, \cite{baake10}). 

\subsection{Stealthy hyperuniform systems}
As alluded to earlier, an important category of processes with spectral constraints   is that of stealthy hyperuniform processes (SHP). Originating in the study of random point fields, these processes pertain to the situation where the  spectrum  vanishes in a neighbourhood of the origin. The nomenclature ``stealthy'' originates from the fact that such a point configuration is invisible to diffraction experiments involving frequencies that fall in the ``spectral gap''. Stealthy hyperuniform processes have been the subject of intensive investigations in the recent past, see e.g. \cite{ToZS}, \cite{ZhST-2}, \cite{ZhST-1}, \cite{ZhST-3}, \cite{ChDZCT}, \cite{GhL-CMP} for a partial list, and the references therein . SHP are naturally hyperuniform,  they are a natural class of models for investigation under the ambit of spectrally constrained stochastic systems. 

In \cite{GhL-CMP}, a rigorous mathematical investigation of SHP was undertaken. In fact, most of the results therein are applicable to a wider class of models, which the authors referred to as generalized stealthy processes. These are translation invariant random fields (or more generally, random measures) whose  spectrum vanishes in some open subset of the frequency domain (significantly relaxing the requirement that the ``spectral gap'' be a neighbourhood of the origin). An important theme of the results in \cite{GhL-CMP} is a very high degree of ``orderliness'' exhibited by SHP. This can be observed, for instance, in the bounded holes conjecture of Torquato, Zhang and Stillinger, which was established in the affirmative in \cite{GhL-CMP}. This result entails that the ``holes'' (i.e., regions in the physical space that are devoid of particles), are at most of a deterministically fixed size. Moreover, it was shown that this deterministic upper bound on hole sizes in inversely proportional to the size of the spectral gap. 

\subsection{Maximal Rigidity and its consequences}
The most intriguing property of generalized stealthy systems, established in \cite{GhL-CMP}, is perhaps the result that such systems exhibit ``maximal rigidity''. That is, the exact configuration (in the case of particle systems) or the exact realisation of the random field / random measure , when restricted to a bounded domain of the physical space, is a deterministic function of the configuration (realisation of the random field) outside the domain. This caps a fairly long line of work on ``rigidity phenomena'' in random point fields (\cite{Gh}, \cite{GhP} \cite{GhL-JSP}, \cite{Bu}, \cite{BuQi}, \cite{NiKi}),  which entails that certain statistics of  local particle configurations  (like local mass, local center of mass, etc) are degenerate (that is, non-random) under spatial conditioning. With the natural understanding that the complete determination (or degeneracy) of the field under spatial conditioning is justifiably referred to as maximal rigidity, SHP form an important class of models from this perspective.

The true physical interpretation or implication of the maximal rigidity exhibited by stealthy systems is not well-understood. 
Rigidity  under spatial conditioning is one of several possible ways to address the question of statistical degeneracy in a spatial system, and it naturally begs the question, exactly how degenerate are stealthy random fields ? While it might be tempting to contemplate a complete or nearly complete lack of randomness (in some appropriate sense), a cautionary example is provided by the class of stealthy Gaussian random fields. It is known from classical Gaussian process theory that thanks to the Bochner-Khinchine Theorem \cite[Theorem 4.2.2]{lukacs}, there is a one-to-one correspondence between translation invariant Gaussian random fields and non-negative spectral measures. Using this dictionary, a Gaussian random field is stealthy as soon as the spectral measure vanishes on some neighbourhood of the frequency domain - a fairly mild condition in the context of the Bochner-Khinchine theorem, thereby ensuring that a vast category of Gaussian random fields are, in fact, stealthy. However, the mildness of this constraint in view of Bochner-Khinchine also guarantees at the same time that these processes can be hardly viewed to be devoid of or lacking in randomness in any significant sense. 

\subsection{Spectral constraints and measuring the lack of randomness}
In view of these considerations, the question of describing the nature of the ``lack of randomness'' in stealthy processes becomes an intriguing and challenging one. In this work, we investigate various aspects of stealthy random fields that touch upon this question. More generally, we  extend our investigation to spectrally constrained random fields that are not stealthy but exhibit hyperuniformity in the sense of a vanishing  spectrum at some point (usually the origin in the frequency space). Taking a refined view-point, we investigate the degree of such spectral vanishing, particularly with regard to  its consequences for statistical constraints on the random field in the physical space. In our investigations, we focus attention on two classical measures of orderliness in  random processes, namely fluctuations and entropy. 

\section{Comparison of fluctuations}\label{sec:explanation}
Let $A_{\delta}=\{x\in \R\; :\; \mbox{there exists $i$ such that  $|x_i|<\delta$}\}$ for $\delta>0$. Suppose $S(x)=0$ when $x\in A_{\delta}$ for some $\delta>0$, and $S$ is bounded function in $[-\pi,\pi]^d$. Then it is clear that  $\int \frac{S(x)}{x_1^2\cdots x_d^2}dx<\infty.$ Therefore Theorem \ref{lem:covarianced} implies that 
\begin{align}\label{eqn:box}
	\lim_{L\to\infty}\Var(Q_{ \CC_{L}}(X))<\infty.
\end{align}
On the other hand, the same assumption on $S$ implies that there exists $L_0$ such that  $\int_{\|\xi\|\le \frac{c}{L}}S(\xi)d\xi=0$ and $\int_{\|\xi\|>\frac{c}{L}}\frac{S(\xi)}{\|\xi\|^{d+1}}d\xi<\infty$ for $L>L_0$. Therefore Proposition \ref{lem:lowervariance} implies that
\begin{align}\label{eqn:ball}
	\Var(Q_{\BB_L}(X))\gtrsim  L^{d-1}.
\end{align}
In this section we give an intuitive explanation for the different behaviour observed in \eqref{eqn:box} and \eqref{eqn:ball} for $d=2$. See Figure \ref{fig;1}.

\begin{figure}
\begin{center}
	\includegraphics[width=0.7\textwidth]{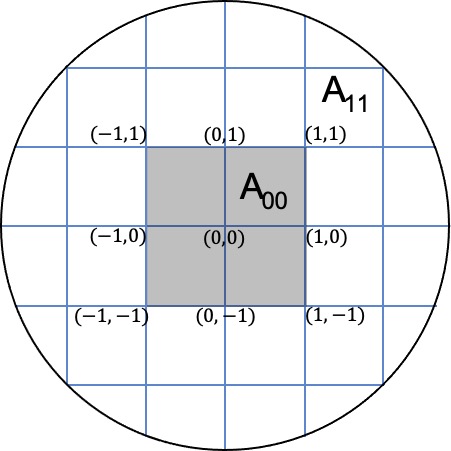}
	\caption{Each box in the shaded area is surrounded by $8$-complete boxes, other boxes do not satisfy this condition.}\label{fig;1}
\end{center}
\end{figure}

To see this phenomenon, we first divide the ball of radius of $L$ in grids of length $\sqrt L$, as shown in the figure. We denote
\[
A_{(k,l)}=[k\sqrt L, (k+1)\sqrt L]\times [\ell \sqrt L, (\ell+1)\sqrt L], \mbox{ for $k,\ell \in \Z$}.
\]
Note that $A_{(k,\ell)}$ are  squares with side lengths $\sqrt L$.  Let $\vp_L=\one_{\BB_L}$. Then 
\begin{align*}
	Q_{\BB_{L}}(X)=\sum_{k,\ell}Y_{k,\ell},\mbox{ where $Y_{k,\ell}=\sum\limits_{i\in A_{k,\ell}\cap \BB_L}X_i$}.
\end{align*}
If $A_{k,\ell}\cap \BB_L=\emptyset$ then $Y_{k,\ell}=0$. Therefore 
\begin{align*}
	Q_{\BB_{L}}(X)^2=\sum_{k,\ell}\sum_{p,q}Y_{k,\ell}Y_{p,q}.
\end{align*}
Since $\E[	Q_{\BB_{L}}(X)]=0$, we have
\begin{align*}
	\Var(Q_{ \BB_{L}}(X))=\sum_{k,\ell}\sum_{p,q}\E[Y_{k,\ell}Y_{p,q}].
\end{align*}
Now consider the term when $k=0$ and $\ell=0$, i.e.,
$
\sum_{p,q}\E[Y_{0,0}Y_{p,q}].
$
Roughly speaking, Theorem \ref{lem:covarianced} implies that, for large $L$, 
\[
\E[Y_{0,0}Y_{p,q}]\approx  0\mbox{ if $\|(0,0),(p,q)\|_{\infty}\ge 2$.}
\]
Therefore  we have
\begin{align*}
	\sum_{p,q}\E[Y_{0,0}Y_{p,q}]&\approx\sum_{|p|, |q|\le 1} \E[Y_{0,0}Y_{p,q}]
	\\&=\E[Y_{0,0}Y_{0,0}]+4\E[Y_{0,0}Y_{0,1}]+4\E[Y_{0,0}Y_{1,1}]
	\\&=\sigma^2-4\frac{\sigma^2}{2}+4\frac{\sigma^2}{4}=0.
\end{align*}
Let $Z_0=\{(k,\ell)\in \Z^2 \; :\;  A_{k,\ell} \mbox{ is surrounded by $8$ complete boxes}\}$.  For example in Figure \ref{fig;1}, each shaded box in the figure is surrounded by $8$ complete boxes. But, the boxes which are not shaded do not satisfy this condition. Then, for  $(k,\ell)\in Z_0$, we have
\begin{align*}
	\sum_{p,q}\E[Y_{k,\ell}Y_{p,q}]\approx 0.
\end{align*}
Let $Z_1=\{(k,\ell)\in \Z^2\; :\; A_{k,\ell}\cap \BB_L\neq \emptyset\}$. It is clear that if $(k,\ell)\in Z_1\backslash Z_0$ then 
\[
\sum_{p,q}\E[Y_{k,\ell}Y_{p,q}]\neq 0.
\]
Again observe that $|Z_1\backslash Z_0|\approx L$. Which shows that $\Var(	Q_{\BB_{L}}(X))$ is growing like $L$, not constant.

Carrying out a similar analysis for an $L\times L$ square shows that only the corner $\sqrt L \times \sqrt L $ squares contribute. Thus  $\Var(Q_{\CC_L})=4\times \frac{\sigma^2}{4}=\sigma^2$. The same is true for rectangles with sides of lengths proportional to $L$.

\section{Proof of Theorem \ref{thm:varianceball}}
This section is dedicated to prove Theorem \ref{thm:varianceball}. The following lemma will be used repeatedly in the paper.

\begin{lemma}\label{ft:formula}
	Let $X=(X_i)_{i\in \Z^d}$ be a translation invariant random field with covariance kernel $K$. Suppose $S$ is  the Fourier transform of the covariance kernel $K$. Then 
	\begin{align*}
		\Var(\vp(X))&=\frac{1}{(2\pi)^d}\int {|\hat{\vp}(\zeta)}|^2S(\zeta)d\zeta \;\mbox{  and }\; \\\cov(\vp(X),\psi(X))&=\frac{1}{(2\pi)^d}\int {\hat{\vp}(\zeta)}\bar{\hat\psi(\zeta)}S(\zeta)d\zeta,
	\end{align*}
	where $\vp$ and $\psi$ are two functions in  $(L^1\cap L^2)(\R^d)$, and $\vp(X)=\sum_{i\in \Z^d}\vp(i)X_i$. The $\hat \vp$ and $\hat \psi$ are the Fourier transforms of $\vp$ and $\psi$ respectively.
\end{lemma}

\begin{proof}[Proof of Lemma \ref{ft:formula}]
	Recall $\vp(X)=\sum_{i\in Z^d}\vp(i)X_i$. Therefore we have
	\begin{align*}
		\Var(\vp(X))=\sum_{p,q}\vp(p)\overline{\vp(q)}\E[X_pX_q]=\sum_{p,q}\vp(p)\overline{\vp(q)}K(p-q).
	\end{align*}
	By the definition of $S$ we have $K(p)=\hat S(p)$ for $p\in \Z^d$, i.e., 
	\begin{align*}
		K(p)=\frac{1}{(2\pi)^d}\int_{[-\pi,\pi]^d} S(\zeta)e^{-i\zeta\cdot p}d\theta.
	\end{align*}
	Therefore we get
	\begin{align*}
		\Var(\vp(X))&=\frac{1}{(2\pi)^d}\sum_{p,q}\vp(p)\overline{\vp(q)}\int_{[-\pi,\pi]^d} S(\zeta)e^{-i\zeta\cdot (p-q)}d\zeta
		\\&=\frac{1}{(2\pi)^d}\int_{[-\pi,\pi]^d}S(\zeta)\sum_{p,q}\vp(p)\overline{\vp(q)}e^{-i\zeta\cdot (p-q)}d\zeta
		\\&=\frac{1}{(2\pi)^d}\int_{[-\pi,\pi]^d}S(\zeta)\l(\sum_{p}\vp(p)e^{-i\zeta\cdot p}\r)\l(\overline{\sum_{q}\vp(q)e^{-i\zeta\cdot q}}\r)d\zeta
		\\&=\frac{1}{(2\pi)^d}\int S(\zeta)|\hat\vp(\zeta)|^2d\zeta.
	\end{align*}
	Similarly it can be shown, we skip the details here, that 
	\begin{align*}
		\cov(\vp(X),\psi(X))=\frac{1}{(2\pi)^d}\int \hat\vp(\zeta)\overline{\hat\psi(\zeta)}S(\zeta)d\zeta.
	\end{align*}
	Hence the result.
\end{proof}

We write $f(L)\lesssim g(L)$ for all $L$ if and only if there exists a constant $C$ such that $f(L)\le C g(L)$ for all $L$.  We have the following result. 
\begin{proposition}\label{lem:varianceforball}
	For large $L$,
	\begin{align*}
	\Var(Q_{\BB_{L}}(X))\lesssim L^{2d}\int_{\|\xi\|\le \frac{c}{L}}S(\xi)d\xi + L^{d-1}\int_{\|\xi\|>\frac{c}{L}}\frac{S(\xi)}{\|\xi\|^{d+1}}d\xi,
	\end{align*}
	for some positive constant $c$.
\end{proposition}
Note that there is no assumption on $S$ for the upper bound.
The next result shows that the upper bound is tight under a mild condition on $S$. 
\begin{proposition}\label{lem:lowervariance}
	Suppose $S$ satisfies the condition as in Theorem \ref{thm:varianceball}.	Then, 
	\begin{align*}
	\Var(Q_{\BB_{L}}(X))\gtrsim L^{2d}\int_{\|\xi\|\le \frac{c}{L}}S(\xi)d\xi + L^{d-1}\int_{\|\xi\|>\frac{c}{L}}\frac{S(\xi)}{\|\xi\|^{d+1}}d\xi, \mbox{ for large $L$,}
	\end{align*}
	for some positive constant $c$.
\end{proposition}
\begin{proof}[Proof of Theorem \ref{thm:varianceball}] 
	The result follows from Propositions \ref{lem:varianceforball} and \ref{lem:lowervariance}.
\end{proof}

The rest of this section is dedicated to prove Propositions \ref{lem:varianceforball} and \ref{lem:lowervariance}. The Bessel functions play a crucial role in proving the propositions. We first  recall the definition of Bessel functions.  If $\Re(\nu)>-\frac{1}{2}$ then $J_{\nu}(z)$, the {\it  Bessel function of order $\nu$},  is defined (see \cite[p. 128]{epstein}) by the integral 
\begin{align}\label{eqn:besselfn}
J_{\nu}(z)=\frac{(\frac{z}{2})^{\nu}}{\Gamma(\nu+\frac{1}{2})\Gamma(\frac{1}{2})}\int_{0}^{\pi}e^{iz\cos \theta}\sin^{2\nu}(\theta)d\theta.
\end{align}

%

\begin{proposition}\label{prop:fouriertransformofball}
	{\it Suppose $\vp_{L}=\one_{\BB_{L}}$. Then
		\begin{align}\label{eqn:reducebessel}
		\hat{\vp}_L(\xi)=c_d'L^d\int_{0}^{\pi}e^{iL\|\xi\|\cos \theta}\sin^{d}(\theta )d\theta,
		\end{align}
		where $c_d'$ is a constant depending on $d$.}
\end{proposition}

\begin{proof}[Proof of Proposition \ref{prop:fouriertransformofball}]
Observe that $\vp_L(x)=\vp_1(x/L)$. Therefore we have
\begin{align}\label{eqn:vpl}
	\widehat{\vp_{L}}(\xi)=\widehat{\vp_1(x/L)}(\xi)=L^d\widehat{\vp_1}(L\xi).
\end{align}
Again,	the Fourier transform of  the characteristic function of the unit ball $\mathcal B_1\subset \R^d$  is given  by, see \cite[Example 4.5.3]{epstein},
	\begin{align}\label{eqn:formula1}
	\hat{\vp}_1(\xi)
	=\frac{c_d}{\|\xi\|^{\frac{d}{2}}}J_{\frac{d}{2}}(\|\xi\|).
	\end{align}
	  Using \eqref{eqn:formula1} in \eqref{eqn:vpl} we get
	\begin{align}\label{eqn:fourierinverseballL}
	\hat{\vp}_L(\xi)=\frac{c_dL^{\frac{d}{2}}}{\|\xi\|^{\frac{d}{2}}}J_{\frac{d}{2}}(L\|\xi\|).
	\end{align}
	Then from \eqref{eqn:besselfn} we get 
	\begin{align}\label{eqn:reduction1}
	J_{\frac{d}{2}}(L\|\xi\|)=\frac{(\frac{L\|\xi\|}{2})^{\frac{d}{2}}}{\Gamma(\frac{d}{2}+\frac{1}{2})\Gamma(\frac{1}{2})}\int_{0}^{\pi}e^{iL\|\xi\|\cos \theta}\sin^{d}(\theta )d\theta.
	\end{align}
	The result follows from \eqref{eqn:fourierinverseballL} and \eqref{eqn:reduction1}.
\end{proof}


We also use the following asymptotic of Bessel functions.

\subsection{Asymptotic behavior of Bessel functions :} We have, see \cite[p. 364, 9.2.1]{handbook},
\begin{align*}
J_{\alpha}(z)=\sqrt{\frac{2}{\pi z}}\l(\cos(z-\frac{\alpha \pi}{2}-\frac{\pi}{4})+e^{|Im (z)|}O(|z|^{-1})\r),\mbox{ for $|$arg$(z)|<\pi$}.
\end{align*}
In particular if $Im(z)=0$ then we have
\begin{align*}
	J_{\alpha}(z)=\sqrt{\frac{2}{\pi z}}\l(\cos(z-\frac{\alpha \pi}{2}-\frac{\pi}{4})+O(|z|^{-1})\r),\mbox{ for $z\in \R$}.
\end{align*}
Then there exists a large $c>0$ such that
\begin{align}\label{eqn:asymptoticofbesssel}
J_{\alpha}(z)=\sqrt{\frac{2}{\pi z}}\l(\cos(z-\frac{\alpha \pi}{2}-\frac{\pi}{4})+g(z)\r), \mbox{ for $z\in \R$}.
\end{align}  
where $|g(z)|\lesssim |z|^{-1}$ for all $|z|>c$.

Now we proceed to prove Propositions \ref{lem:varianceforball} and \ref{lem:lowervariance}.

\begin{proof}[Proof of Proposition \ref{lem:varianceforball}]

	Let $\vp_{L}=\one_{\BB_L}$. By  Lemma \ref{ft:formula} we have 
	\begin{align*}
	\Var(Q_{\BB_{L}}(X))&=\frac{1}{(2\pi)^d}\int |\hat{\vp}_L(\xi)|^2S(\xi)d\xi
	\\&=\frac{1}{(2\pi)^d}\int_{\|\xi\|\le \frac{c}{L}}|\hat{\vp}_L(\xi)|^2S(\xi)d\xi+\frac{1}{(2\pi)^d}\int_{\|\xi\|> \frac{c}{L}}|\hat{\vp}_L(\xi)|^2S(\xi)d\xi.
	\end{align*}
	Note that from \eqref{eqn:reducebessel} we have
	\begin{align*}
	|\hat{\vp}_L(\xi)|^2\lesssim L^{2d}, \mbox{ for $\xi\in [-\pi,\pi]^d$}.
	\end{align*}	
	Therefore we get 
	\begin{align}\label{eqn:upb1}
	\int_{\|\xi\|\le \frac{c}{L}}|\hat{\vp}_L(\xi)|^2S(\xi)d\xi \lesssim L^{2d}\int_{\|\xi\|\le \frac{c}{L}}S(\xi)d\xi.
	\end{align}	
	Choose $c$ such that \eqref{eqn:asymptoticofbesssel} holds. Then, for large $L$, we have
	\begin{align*}
	J_{\frac{d}{2}}(L\|\xi\|)\lesssim \frac{1}{\sqrt{L\|\xi\|}}, \mbox{ for $\|\xi\|>\frac{c}{L}$}.
	\end{align*}	
	Then from \eqref{eqn:fourierinverseballL}, for $\|\xi\|>\frac{c}{L}$, we get
	\begin{align*}
	|\hat{\vp}_L(\xi)|^2\lesssim \frac{L^{d-1}}{\|\xi\|^{d+1}}.
	\end{align*}
	Therefore we have
	\begin{align}\label{eqn:upb2}
	\int_{\|\xi\|> \frac{c}{L}}|\hat{\vp}_L(\xi)|^2S(\xi)d\xi\lesssim L^{d-1}\int_{\|\xi\|> \frac{c}{L}} \frac{S(\xi)}{\|\xi\|^{d+1}}d\xi.
	\end{align}
	The result follows from \eqref{eqn:upb1} and \eqref{eqn:upb2}.
\end{proof}

\begin{proof}[Proof of Proposition \ref{lem:lowervariance}]
	From \eqref{eqn:reducebessel} we have
	\begin{align*}
	|\hat{\vp}_L(\xi)|^2 & \ge c_d'^2L^{2d}\l( \int_{0}^{\pi} \cos(L\|\xi\|\cos \theta)\sin^d(\theta) d\theta\r)^2.
	\end{align*}
	Note that, for $L\|\xi\|\le c$, there exists $\delta_c$ such that 
	\begin{align*}
	L\|\xi\|\cos \theta\le \frac{\pi}{4}, \mbox{  for $\theta\in [\frac{\pi}{2}-\delta_c,\frac{\pi}{2}+\delta_c]$}.
	\end{align*}
	Therefore, for $L\|\xi\|\le c$, we have 
	\begin{align*}
	\cos(L\|\xi\|\cos \theta)\sin^d(\theta)\ge \frac{ \cos^{d}(\delta_c)}{\sqrt 2}, \mbox{ for $\theta\in [\frac{\pi}{2}-\delta_c,\frac{\pi}{2}+\delta_c]$}.
	\end{align*}
	Hence we get 
	\begin{align*}
	|\hat{\vp}_L(\xi)|^2\gtrsim L^{2d},	\mbox{ for $L\|\xi\|\le c$}.
	\end{align*}
	Which implies , as $S\ge 0$, that 
	\begin{align}\label{eqn:LB1}
	\int_{\|\xi\|\le \frac{c}{L}}|\hat{\vp}_L(\xi)|^2S(\xi)d\xi\gtrsim L^{2d}\int_{\|\xi\|\le \frac{c}{L}}S(\xi)d\xi.
	\end{align}
	Now choose $c>0$ such that  \eqref{eqn:asymptoticofbesssel} holds. Then by \eqref{eqn:asymptoticofbesssel} from \eqref{eqn:fourierinverseballL}  we  get
	\begin{align}\label{eqn:LBphi}
	|\hat{\vp}_L(\xi)|^2&=\frac{C_dL^{d-1}}{\|\xi\|^{d+1}}\l(\cos(L\|\xi\|-\frac{d \pi}{4}-\frac{\pi}{4})+g(L\|\xi\|)\r)^2,
	\end{align}
	where $C_d$ is a constant depending on $d$. Note that, as $|g(L\|\xi\|)|\lesssim \frac{1}{L\|\xi\|}$, 
	\begin{align*}
	\l|\int_{\|\xi\|> \frac{c}{L}}\frac{C_dL^{d-1}g(L\|\xi\|)}{\|\xi\|^{d+1}} S(\xi)d\xi \r|\lesssim L^{d-2} \int_{\|\xi\|> \frac{c}{L}}\frac{ S(\xi)}{\|\xi\|^{d+2}}d\xi.
	\end{align*}
	Let $\phi=\frac{d \pi}{4}+\frac{\pi}{4}$. Therefore form \eqref{eqn:LBphi} we have 
	\begin{align}\label{eqn:lb2}
	\int_{\|\xi\|> \frac{c}{L}}|\hat{\vp}_L(\xi)|^2S(\xi)d\xi=&C_dL^{d-1}\mathcal I_1+O(L^{d-2})\mathcal I_2,
	\end{align}
	where 
	$$\mathcal I_1=\int_{\|\xi\|> \frac{c}{L}} \frac{S(\xi)}{\|\xi\|^{d+1}}\cos^2(L\|\xi\|-\phi)d\xi \mbox{ and } \mathcal{I}_2=\int_{\|\xi\|> \frac{c}{L}}\frac{ S(\xi)}{\|\xi\|^{d+2}}d\xi.$$
	The rest of proof is dedicated to estimate $\mathcal I_1$. 
	Let $\xi=\|\xi\|\omega $, where $\omega\in \S^{d-1}$. Then $d\xi=\|\xi\|^{d-1}d\|\xi\|dV_{\S^{d-1}}(\omega)$
	and 
	\begin{align*}
	\mathcal I_1=\int_{\S^{d-1}}\int_{\|\xi\|> \frac{c}{L}} \frac{S(\|\xi\|\omega)}{\|\xi\|^{2}}\cos^2(L\|\xi\|-\phi)d\|\xi\|dV_{\S^{d-1}}(\omega).
	\end{align*}
	Since $S$ is bounded on $[-\pi,\pi]^d$, for fixed $\epsilon>0$, we have
	\begin{align}\label{eqn:constantbound}
	0\le \int_{\S^{d-1}}\int_{\|\xi\|>\epsilon} \frac{S(\|\xi\|\omega)}{\|\xi\|^{2}}\cos^2(L\|\xi\|-\phi)d\|\xi\| dV_{\S^{d-1}}(\omega)\lesssim \frac{1}{\epsilon}.
	\end{align}
	Since $S$ is regular, there exists a function $A$ such that 
	\begin{align*}
	c_1 A(\xi)\le S(\xi)\le c_2 A(\xi) \mbox{ and } \frac{\int_{\S^{d-1}}A(\|\xi\|\omega)dV_{\S^{d-1}}(\omega)}{\|\xi\|^2}
	\end{align*}
	is monotone for $0\le \|\xi\|\le \epsilon$, where $c_1,c_2,\epsilon>0$. Note that $A$ is bounded and positive, as $S$ is bounded and positive. Therefore 
	\begin{align*}
	\mathcal I_1 \gtrsim \int_{\S^{d-1}}\int_{\frac{c}{L}<\|\xi\|<\eps} \frac{A(\|\xi\|\omega)}{\|\xi\|^{2}}\cos^2(L\|\xi\|-\phi)d\|\xi\| dV_{\S^{d-1}}(\omega).
	\end{align*}
	\noindent{\bf Suppose $A$ is radial:} For simplicity, we first assume that $A$ is radial, i.e., $A(\xi)=A(\|\xi\|)$. Then the regularity condition implies that $\frac{A(\|\xi\|)}{\|\xi\|^{2}}$ is monotone for $0\le \|\xi\|\le \epsilon$. Let $\frac{A(\|\xi\|)}{\|\xi\|^{2}}$ be monotone decreasing. Let $k_0\in \N$ be such that $c\le k_0\pi$, and $r_k=\frac{\pi k}{L}, k=1,2,\ldots$. Then
	\begin{align*}
	\int\limits_{ \frac{c}{L}<\|\xi\|<\epsilon} \frac{A(\|\xi\|)}{\|\xi\|^{2}}\cos^2(L\|\xi\|-\phi)d\|\xi\|
	\ge &\sum_{k=k_0}^{\lfloor \frac{\epsilon L}{\pi}\rfloor}\int_{r_k}^{r_{k+1}}\frac{A(\|\xi\|)}{\|\xi\|^{2}}\cos^2(L\|\xi\|-\phi)d\|\xi\|
	\\\ge &\sum_{k=k_0}^{\lfloor \frac{\epsilon L}{\pi}\rfloor}\frac{A(r_{k+1})}{(r_{k+1})^{2}}\int_{r_k}^{r_{k+1}}\cos^2(L\|\xi\|-\phi)d\|\xi\|
	\\= &\sum_{k=k_0}^{\lfloor \frac{\epsilon L}{\pi}\rfloor}\frac{A(r_{k+1})}{(r_{k+1})^{2}}\frac{\pi}{2L}.
	\end{align*}
	The last equality follows from the fact that
	\begin{align*}
	\int_{r_k}^{r_{k+1}}\cos^2(L\|\xi\|-\phi)d\|\xi\|= \frac{\pi}{2L}, \mbox{ for $k=1,2,\ldots $.}
	\end{align*}
	As $\frac{A(\|\xi\|)}{\|\xi\|^{2}}$ is monotone decreasing  for $0\le \|\xi\|\le \epsilon$, by the  Riemann integration 
	\begin{align*}
	\sum_{k=k_0}^{\lfloor \frac{\epsilon L}{\pi}\rfloor}\frac{A(r_{k+1})}{(r_{k+1})^{2}}\frac{\pi}{L} \gtrsim \int_{\frac{c}{L}}^{\epsilon}\frac{A(\|\xi\|)}{\|\xi\|^{2}}d\|\xi\|, \mbox{ for large $L$},
	\end{align*}
	as monotone functions are integrable. Thus,  for large $L$, we have  
	\begin{align}\label{eqn:riemann}
	\mathcal I_1\gtrsim \int\limits_{ \frac{c}{L}<\|\xi\|<\epsilon} \frac{A(\|\xi\|)}{\|\xi\|^{2}}\cos^2(L\|\xi\|-\phi)d\|\xi\|\gtrsim \int_{\frac{c}{L}}^{\epsilon}\frac{A(\|\xi\|)}{\|\xi\|^{2}}d\|\xi\|.
	\end{align}
	Combining \eqref{eqn:lb2}, \eqref{eqn:constantbound} and \eqref{eqn:riemann} we get, for large $L$, 
	\begin{align*}
	\int_{\|\xi\|> \frac{c}{L}}|\hat{\vp}_L(\xi)|^2S(\xi)d\xi \gtrsim L^{d-1} \l(\int_{\frac{c}{L}}^{\epsilon}\frac{A(\|\xi\|)}{\|\xi\|^{2}}d\|\xi\|+O(1)\r).  
	\end{align*}
Since $A$ is bounded and  $A(\xi)=A(\|\xi\|)$,  we get
	\begin{align*}
	\int_{\|\xi\|> \frac{c}{L}}|\hat{\vp}_L(\xi)|^2S(\xi)d\xi \gtrsim L^{d-1} \int_{\|\xi\|> \frac{c}{L}}\frac{A(\xi)}{\|\xi\|^{d+1}}d\xi.
	\end{align*}
	Using the regularity condition of $S$ we have
	\begin{align*}
	\int_{\|\xi\|> \frac{c}{L}}|\hat{\vp}_L(\xi)|^2S(\xi)d\xi\gtrsim L^{d-1} \int_{\|\xi\|> \frac{c}{L}}\frac{S(\xi)}{\|\xi\|^{d+1}}d\xi, \mbox{ for large $L$}.
	\end{align*}
	The inequalities also hold when $\frac{A(\|\xi\|)}{\|\xi\|^{2}}$ is monotone increasing  for $0\le \|\xi\|\le \epsilon$.

	\vspace{.2cm}
	\noindent{\bf  Suppose $A$ is not radial :} 
	Let 
	$ 
	h(\|\xi\|):=\frac{\int_{\S^{d-1}}A(\|\xi\|\omega)dV_{\S^{d-1}}(\omega)}{\|\xi\|^2}.
	$
	Then by change of variables we get
	\begin{align*}
	\int_{\frac{c}{L}<\|\xi\|<\eps } \frac{A(\xi)}{\|\xi\|^{d+1}}\cos^2(L\|\xi\|-\phi)d\xi=\int_{\frac{c}{L}<\|\xi\|<\eps }h(\|\xi\|)\cos^2(L\|\xi\|-\phi)d\|\xi\|.
	\end{align*}
	Now suppose $h$ 
	is decreasing for $0\le \|\xi \| \le \eps$,
	using the same arguments as before it can be shown that
	\begin{align*}
	\int_{\frac{c}{L}<\|\xi\|<\eps }h(\|\xi\|)\cos^2(L\|\xi\|-\phi)d\|\xi\|&\gtrsim \int_{\frac{c}{L} }^\eps h(\|\xi\|)d\|\xi\|
	\\&=\int_{\S^{d-1}}\int_{\frac{c}{L} }^\eps \frac{A(\|\xi\|\omega)}{\|\xi\|^2}d\|\xi\|dV_{\S^{d-1}}(\omega)
	\\&=\int_{\frac{c}{L}<\|\xi\|<\eps }\frac{A(\xi)}{\|\xi\|^{d+1}}d\xi.
	\end{align*}
	We get the last equality by the change of variables. Thus we have
	\begin{align*}
	\mathcal I_1 \gtrsim \int_{\frac{c}{L}<\|\xi\|<\eps }\frac{A(\xi)}{\|\xi\|^{d+1}}d\xi\gtrsim \int_{\frac{c}{L}<\|\xi\|<\eps }\frac{S(\xi)}{\|\xi\|^{d+1}}d\xi.
	\end{align*}
	Therefore, by \eqref{eqn:lb2} and  \eqref{eqn:constantbound}, we get 
	\begin{align}\label{eqn:LB2}
	\int\limits_{\|\xi\|> \frac{c}{L}}|\hat{\vp}_L(\xi)|^2S(\xi)d\xi \gtrsim L^{d-1} \int_{\|\xi\|>\frac{c}{L}}\frac{S(\xi)}{\|\xi\|^{d+1}}d\xi.
	\end{align}
	Therefore \eqref{eqn:LB1} and \eqref{eqn:LB2}  give the result.
\end{proof}

\section{proof of Theorem \ref{lem:covarianced}}
In this section we give the proof of Theorem \ref{lem:covarianced}.

\begin{proof}[Proof of Theorem \ref{lem:covarianced}]
	Let $\vp_{0,L}(x)=\one_{\CC_L^{(0)}}(x)$ and $\vp_{\n, L}(x)=\one_{\CC_L^{(\n)}}(x)$. Then 
	\begin{align*}
		Q_{\CC^{(0)}_L}(X)=\vp_{0,L}(X) \mbox{ and } Q_{\CC^{(\n)}_L}(X)=\vp_{\n, L}(X).
	\end{align*}
 Therefore using Lemma \ref{ft:formula} we have
	\begin{align}\label{eqn:cov1}
	\E[Q_{\CC^{(0)}_L}(X)Q_{\CC^{(\n)}_L}(X)]=\frac{1}{(2\pi)^d}\int_{-\pi}^{\pi}\cdots \int_{-\pi}^{\pi}{\hat\vp_{0,L}}(x)\bar{{\hat\vp_{\n, L}}(x)}S(x)dx,
	\end{align}
where $dx=dx_1\cdots dx_d$. Observe that $\vp_{\n,L}(x)=\vp_{0,L}(x-\n L)$. Therefore, as $\widehat{f(\cdot -a)}=e^{-i\langle \cdot, a\rangle}\widehat{f}(\cdot)$, we have 
\begin{align*}
\widehat{\vp_{\n,L}}(x)=\widehat{\vp_1(\cdot-\n L)}(x)=e^{-i\langle \n L, x\rangle}\widehat{\vp_{0,L}}(x)=\prod_{k=1}^{d}\frac{(1-e^{-iLx_k})e^{-iLn_kx_k}}{ix_k},
\end{align*}
as the Fourier transform of $\vp_{0,L}$ is given by, for $x\in \R^d$, 
	\begin{align*}
	\widehat{\vp_{0,L}}(x)&=\int \cdots \int e^{-ix\cdot t}\one_{\CC_{L}^{(\n)}}(t)dt
	=\prod_{k=1}^{d}\int_{0}^{L}e^{-ix_kt_k}dt_k=\prod_{k=1}^{d}\frac{(1-e^{-iLx_k})}{ix_k}.
	\end{align*}
	Therefore, for $x\in \R^d$ and $\n\in (\N\cup \{0\})^d$, we get
	\begin{align*}
	{\hat\vp_{0,L}}(x)\bar{{\hat\vp_{\n, L}}(x)}=\frac{\prod_{k=1}^{d}(1-e^{-iLx_k})(1-e^{iLx_k})e^{iLn_kx_k}}{x_1^2\cdots x_d^2}.
	\end{align*}
	Let $\epsilon_{k,1}^{(1)}=0, \epsilon_{k,1}^{(2)}=1, \epsilon_{k,2}^{(1)}=0, \epsilon_{k,2}^{(2)}=-1 $ for $k=1,\ldots, d$. 
	Then we can write
	\begin{align*}
		(1-e^{-iLx_k})(1-e^{iLx_k})=\sum_{p_k,q_k\in \{1,2\}}(-1)^{\l(\epsilon_{k,1}^{(p_k)}+\epsilon_{k,2}^{(q_k)}\r)}e^{iL\l(\epsilon_{k,1}^{(p_k)}+\epsilon_{k,2}^{(q_k)}\r)},
	\end{align*}
	for $k=1,\ldots, d$. Let $p=(p_1,\ldots,p_d), q=(q_1,\dots, q_d)\in \{1,2\}^d$. Then 
	\begin{align}\label{eqn:cov2}
&	{\hat\vp_{0,L}}(x)\bar{{\hat\vp_{\n, L}}(x)}\nonumber
\\=&\frac{1}{x_1^2\ldots x_d^2}\sum_{p,q}(-1)^{\sum_{k=1}^{d}\l(\epsilon_{k,1}^{(p_k)}+\epsilon_{k,2}^{(q_k)}\r)}e^{iL\sum_{k=1}^{d}\l(\epsilon_{k,1}^{(p_k)}+\epsilon_{k,2}^{(q_k)}+n_k\r)x_k}.
	\end{align}
	The sum is taken over all possible values of $p$ and $q$. Let  $f(x)=\frac{S(x)}{x_1^2\cdots x_d^2}$. The assumption on $S$ implies that $\int_{\R^d} |f(x)|dx<\infty$. Therefore 
	\begin{align*}
		&\frac{1}{(2\pi)^d}\int_{[-\pi,\pi]^d}e^{iL\sum_{k=1}^{d}\l(\epsilon_{k,1}^{(p_k)}+\epsilon_{k,2}^{(q_k)}+n_k\r)x_k}f(x)dx
		\\=&\hat{f}\l(-L(\epsilon_{1,1}^{(p_1)}+\epsilon_{1,2}^{(q_1)}+n_1,\ldots, \epsilon_{d,1}^{(p_d)}+\epsilon_{d,2}^{(q_d)}+n_d)\r),
	\end{align*}
	where $\hat f$ denotes the Fourier transform of $f$. Thus, from \eqref{eqn:cov1} and \eqref{eqn:cov2}, 
	\begin{align*}
	&\E[Q_{\CC^{(0)}_L}(X)Q_{\CC^{(\n)}_L}(X)]
		\\=&\sum_{p,q}(-1)^{\sum_{k=1}^{d}\l(\epsilon_{k,1}^{(p_k)}+\epsilon_{k,2}^{(q_k)}\r)}\hat{f}\l(-L(\epsilon_{1,1}^{(p_1)}+\epsilon_{1,2}^{(q_1)}+n_1,\ldots, \epsilon_{d,1}^{(p_d)}+\epsilon_{d,2}^{(q_d)}+n_d)\r).
	\end{align*}
If $(\epsilon_{1,1}^{(p_1)}+\epsilon_{1,2}^{(q_1)}+n_1,\ldots, \epsilon_{d,1}^{(p_d)}+\epsilon_{d,2}^{(q_d)}+n_d)\neq 0$ then  Lemma \ref{ft:RL} implies that 
	\begin{align}\label{eqn:cov3}
	\lim_{L\to\infty}\hat{f}\l(-L(\epsilon_{1,1}^{(p_1)}+\epsilon_{1,2}^{(q_1)}+n_1,\ldots, \epsilon_{d,1}^{(p_d)}+\epsilon_{d,2}^{(q_d)}+n_d)\r)=0.
	\end{align}
	Suppose $\CC_L^{(0)}$ and $\CC_L^{(\n)}$ are disjoint. Then there exists $k\in \{1,\ldots, d\}$ such that $n_k\ge 2$.
	In this case we have
	$$
		\epsilon_{k,1}^{(p_k)}+\epsilon_{k,2}^{(q_k)}+n_k\ge 1.
	$$	
Therefore \eqref{eqn:cov3} implies that if 	$\CC_L^{(0)}$ and $\CC_L^{(\n)}$ are disjoint then 
\begin{align*}
	\lim_{L\to\infty}\E[Q_{\CC^{(0)}_L}(X)Q_{\CC^{(\n)}_L}(X)]=0.
\end{align*}
	It remains to show that if dim$(\CC_L^{(0)}\cap \CC_L^{(\n)})=d-j$, for $j=0,1, \ldots, d$, then 
	\[
		\lim_{L\to\infty}\E[Q_{\CC^{(0)}_L}(X)Q_{\CC^{(\n)}_L}(X)]=(-1)^j\frac{\sigma_d^2}{2^j}.
	\]
	Let $\n\in (\N\cup \{0\})^d$ such that $(\epsilon_{1,1}^{(p_1)}+\epsilon_{1,2}^{(q_1)}+n_1,\ldots, \epsilon_{d,1}^{(p_d)}+\epsilon_{d,2}^{(q_d)}+n_d)= 0$ for some $p, q\in \{1,2\}^d$. 
	In this case, Lemma \ref{ft:RL} implies that 
	\begin{align}\label{eqn:cov4}
		\lim_{L\to\infty}\E[Q_{\CC^{(0)}_L}(X)Q_{\CC^{(\n)}_L}(X)]=\sum_{p,q}(-1)^{\sum_{k=1}^{d}\l(\epsilon_{k,1}^{(p_k)}+\epsilon_{k,2}^{(q_k)}\r)}\hat{f}(0).
	\end{align}
	Now we need to find the cardinality of the following set
	\begin{align*}
	C_{\n}=\{(p,q)\in \{1,2\}^d\; :\; \epsilon_{k,1}^{(p_k)}+\epsilon_{k,2}^{(q_k)}+n_k=0, k=1,\ldots, d\}.
	\end{align*}
	 Note that if $n_k=0$ then $\epsilon_{k,1}^{(p_k)}+\epsilon_{k,2}^{(q_k)}+n_k=0$ when $(p_k,q_k)\in \{(1,1), (2,2)\}$, and if $n_k=1$ then  $\epsilon_{1,1}^{(p_k)}+\epsilon_{1,2}^{(q_k)}+n_k=0$ when $(p_k,q_k)=(1,2)$. 
	 
	 Observe that  if dim$(\CC_L^{(0)}\cap \CC_L^{(\n)})=d-j$ then there exists $j$ indices $i_1,\ldots, i_{j}$ such that $n_{i_k}=1$ for $k=1,\ldots, j$ and the rest of the $d-j$ coordinates of $\n$ are $0$. Therefore the  cardinality of $C_{\n}$ is
	 $$
	 |C_{\n}|=2^{d-j}.
	 $$
	 	Since $\hat f(0)=\frac{1}{(2\pi)^d}\int_{[-\pi,\pi]^d} \frac{S(x)}{x_1^2\cdots x_d^2}dx=\frac{\sigma_d^2}{2^d}$. Therefore from \eqref{eqn:cov4} we get 
	\[
	\lim_{L\to\infty}\E[Q_{\CC^{(0)}_L}(X)Q_{\CC^{(\n)}_L}(X)]=(-1)^j\frac{\sigma_d^2}{2^j}.
	\]
	The factor $(-1)^j$ appears because $j$ many coordinates of $\n$ are $1$. 
\end{proof}

\section{Proofs of Theorem \ref{lem:variancenonzerodge2} and Corollary \ref{prop:d=2}}\label{sec:dge2nonzeros}
In this section we prove Theorem \ref{lem:variancenonzerodge2} and Corollary \ref{prop:d=2}.
\begin{proof}[Proof of Theorem \ref{lem:variancenonzerodge2}]
We have  $\vp_L=\one_{\CC_L}$. Then 
\begin{align*}
	\hat \vp_L(x)=\int_{-L}^{L}\cdots \int_{-L}^{L}e^{-it\cdot x}dt_1\cdots dt_d=\prod_{k=1}^{d}\frac{e^{iLx_k}-e^{-iLx_k}}{ix_k}=\prod_{k=1}^{d}\frac{2\sin Lx_k}{x_k}.
\end{align*}
Note that $Q_{\CC_L}(X)=\vp_L(X)$. Therefore by Lemma \ref{ft:formula} we have
\begin{align*}
\Var(Q_{\CC_L}(X))&=\frac{1}{(2\pi)^d}\int_{-\pi}^{\pi}\cdots \int_{-\pi}^{\pi}|\hat\vp_L(x)|^2S(x)dx
\\&=\frac{1}{(2\pi)^d}\int_{-\pi}^{\pi}\cdots \int_{-\pi}^{\pi}\prod_{k=1}^{d}\frac{4\sin^2Lx_k}{x_k^2}S(x)dx,
\end{align*}
where $x=(x_1,\ldots, x_d)$ and $dx=dx_1\cdots dx_d$.  As $S$ is symmetric, then
\begin{align}\label{variance}
\Var(Q_{\CC_L}(X))=8^dJ, \mbox{ where $J=\frac{1}{(2\pi)^d}\int_{0}^{\pi}\cdots \int_{0}^{\pi}\prod_{k=1}^{d}\frac{\sin^2Lx_k}{x_k^2}S(x)dx$}.
\end{align}
Let $A_{i_1,\ldots,i_k}^{\delta}=\{ x\in [0,\pi]^d :  |x_i| < \delta \;\mathrm{ iff }\; i \in  \{i_1,\ldots,i_k\} \}$ and $[d]=\{1,\ldots, d\}$ Then define 
$$
A_{\delta}=\cup_{k=1}^d\cup_{\{i_1,\ldots, i_k\}\subset [d]}A_{i_1.\ldots,i_k}^{\delta}.
$$  
Note that $A_\delta^c=[0,\pi]^d\backslash A_{\delta}=\{x\in [0, \pi]^d \; :\; x_1,\ldots, x_d\ge \delta  \}$. Then 
\begin{align*}
	&J=J_1+J_2, \mbox{ where }
\\&	\mbox{$J_1=\frac{1}{(2\pi)^d}\int_{A_\delta}\prod_{k=1}^{d}\frac{\sin^2Lx_k}{x_k^2}S(x)dx$, $J_2=\frac{1}{(2\pi)^d}\int_{A_\delta^c}\prod_{k=1}^{d}\frac{\sin^2Lx_k}{x_k^2}S(x)dx$}.
\end{align*}
Since $\sin^2Lx\le 1$, by $(C2)$ we have
\begin{align}\label{eqn:J2}
	|J_2|<\infty, \mbox{ for $\delta>0$}.
\end{align}
Now we estimate $J_1$. We have 
\[
J_1\le \sum_{k=1}^d\sum_{1\le i_1,\ldots, i_k\le d} I_{i_1,\ldots, i_k}, \mbox{ where } I_{i_1,\ldots, i_k}=\frac{1}{(2\pi)^d}\int_{A_{i_1.\ldots,i_k}^{\delta}}\prod_{k=1}^{d}\frac{\sin^2Lx_k}{x_k^2}S(x)dx.
\]
The assumption on $S$ implies that there exists $c_1,c_2, \delta>0$ such that 
\begin{align*}
c_1x_{i_1}^{\alpha_{i_1}}\ldots x_{i_k}^{\alpha_{i_k}}\le S(x)\le c_2x_{i_1}^{\alpha_{i_1}}\ldots x_{i_k}^{\alpha_{i_k}}
, \mbox{ for $x\in A_{i_1.\ldots,i_k}^{\delta}$}.
\end{align*}
Using the last equation we get
\begin{align*}
&c_1 J_{i_1,\ldots,i_k}'\le 	I_{i_1,\ldots, i_k}\le c_2 J_{i_1,\ldots,i_k}', 
\end{align*}
with
\begin{align*}
J_{i_1,\ldots,i_k}'&=\prod_{j=1}^{k}\int_0^{\delta}\frac{\sin^2Lx_j}{x_j^{2-\alpha_{i_j}}}dx_j\prod_{j={k+1}}^{d}\int_{\delta}^{\pi}\frac{1}{x_j^{2}}dx_j
	=c\prod_{j=1}^{k}\int_0^{\delta}\frac{\sin^2Lx_j}{x_j^{2-\alpha_{i_j}}}dx_j.
\end{align*}
where $c=(\int_{\delta}^{\pi}x^{-2}dx)^{d-k}$ is a constant. Now we show the following lemma.
\begin{lemma}\label{lem:d=1}
Given $\delta>0$, as $L\to \infty$,
\[
\int_{0}^{\delta L}\frac{\sin^2x}{x^{2-\alpha}}dx=\l\{\begin{array}{lcl}
\Theta(1) & \mbox{ if } & \alpha\in [0,1),
\\ \\
\Theta(\log L) & \mbox{ if } & \alpha=1.
\end{array}
\r.
\]
\end{lemma}
\begin{proof}[Proof of Lemma \ref{lem:d=1}]
	It is clear that, as $|\sin x/x|\le 1$ and $|\sin x|\le 1$, 
	\[
	\int_{0}^{\delta L}\frac{\sin^2 x}{x^{2-\alpha}}dx<\int_{0}^{1}x^\alpha dx +\int_1^{\delta L}x^{-2+\alpha}dx.
	\]
	The first integral is obviously bounded, whereas the second is bounded for $\alpha<1$ and behaves like $\log L$ for $\alpha=1$. Hence the lemma. 
\end{proof}
Now observe that by a change of variables and Lemma \ref{lem:d=1} it follows that 
\[
\int_0^{\delta}\frac{\sin^2Lx_j}{x_j^{2-\alpha_{i_j}}}dx_j=\l\{\begin{array}{lcl}
\Theta(1) & \mbox{ if } & \alpha\in [0,1),
\\ \\
\Theta(\log L) & \mbox{ if } & \alpha=1.
\end{array}
\r.
\]
Let $\tau_{i_1,\ldots,i_k}=|\{j\in [k]\; :\; \alpha_{i_j}=1\}|$.
 Then it follows that
\begin{align*}
I_{i_1,\ldots, i_k}=\Theta(L^{\sum_{j=1}^{k}(1-\alpha_{i_j})}(\log L)^{\tau_{i_1,\ldots, i_k}})
\end{align*}
 Clearly, $\sum_{j=1}^{k}(1-\alpha_{i_j})\le m_d$ and  $\tau_{i_1,\ldots,i_d}\le \tau_{d}$ so that $I_{i_1,\ldots, i_d}\le I_{1,\ldots,d}$ and 
\[
I_{1,\ldots, d}=\Theta(L^{d-m_{d}}(\log L)^{\tau_{d}}).
\]
Therefore we have 
\begin{align}\label{eqn:d}
I_{1,\ldots, d}\lesssim J_1\lesssim I_{1,\ldots, d} \imply J_1=\Theta(L^{d-m_{d}}(\log L)^{\tau_{d}}).
\end{align}
The left hand side follows from the fact that $A_{1,\ldots,d}^{\delta}\subset A_{\delta}$. Therefore \eqref{eqn:J2} and \eqref{eqn:d} give the result.
\end{proof}

\begin{proof}[Proof of Corollary \ref{prop:d=2}]
	Let $\vp_L=\one_{[-1,1]^d}$. Then from \eqref{variance} we have
	\begin{align*}
	\Var(Q_{\CC_L}(X))=8^dJ, \mbox{ where $J=\frac{1}{(2\pi)^d}\int_{0}^{\pi}\cdots \int_{0}^{\pi}\prod_{k=1}^{d}\frac{\sin^2Lx_k}{x_k^2}S(x)dx$}.
	\end{align*}
	Let $B(0,\delta):=\{x\in [0,\pi]^d\; :\; \|x\|_2<\delta \}$. Observe that $$
	\{x\in [0,\pi]^d\;: \; (\forall k=1,\ldots, d)\;\; |x_k|\le \delta d^{-1/2}\}\subset B(0,\delta)\subset [-\delta,\delta]^d.
	$$ 
	Then we have $J=J_1+J_2$, where 
	\[
	J_1=\frac{1}{(2\pi)^d}\int_{B(0,\delta)}\prod_{k=1}^{d} \frac{\sin^2Lx_k}{x_k^2}S(x)dx, \; J_2=\frac{1}{(2\pi)^d}\int_{B(0,\delta)^c}\prod_{k=1}^{d}\frac{\sin^2Lx_k}{x_k^2}S(x)dx.
	\]
	By the assumption of $S$ implies that 
	\begin{align*}
		|J_2|<\infty.
	\end{align*}
	So it remains to estimate $J_1$. Since ${S(x)}=\Theta(\|x\|_2^{\alpha}) $ as $\|x\|_2\to 0$, there exist $c_1,c_2,\delta>0$ such that
	\begin{align*}
	c_1\|x\|_2^{\alpha}\le	S(x)\le c_2\|x\|_2^{\alpha}, \mbox{ for $x\in B(0,\delta)$}.
	\end{align*}
Note that $\|x\|_2^{\alpha}\ge |x_1|^{\alpha}$. Therefore we have 
\[
J_1\ge \int_{B(0,\delta)}\prod_{k=1}^{d} \frac{\sin^2Lx_k}{x_k^2}dx_k\gtrsim \frac{1}{(2\pi)^d} \prod_{k=1}^{d}\int_0^{\delta/d^{1/2}}\frac{\sin^2Lx_k}{x_k^{2-\alpha_k}}dx_k
\]
where $\alpha_1=\alpha$ and $\alpha_2=\cdots=\alpha_d=0$. Therefore by Lemma \ref{lem:d=1} we get
\begin{align}\label{eqn:J1low}
J_1\gtrsim \l\{\begin{array}{ll}
L^{d-\alpha} & \mbox{ if $0\le \alpha<1$,}\\ \\
L^{d-1}\log L  & \mbox{ if $\alpha=1$}.
\end{array}\r.
\end{align}
 Again, observe that
\begin{align*}
	\|x\|_2^\alpha\le d^{\alpha/2}(\max\{|x_1|,\ldots, |x_k|\})^{\alpha}\le d^{\alpha/2}(|x_1|^{\alpha}+\cdots+|x_d|^{\alpha}).
\end{align*}
Therefore by Lemma \ref{lem:d=1}, for  $\alpha_j=\alpha$ and $\alpha_k=0$ for $k\neq j$,  we get
\begin{align}\label{eqn:J1up}
J_1\lesssim \prod_{k=1}^{d}\int_0^{\delta}\frac{\sin^2Lx_k}{x_k^{2-\alpha_k}}dx_k\lesssim \l\{\begin{array}{ll}
L^{d-\alpha} & \mbox{ if $0\le \alpha<1$,}\\ \\
L^{d-1}\log L  & \mbox{ if $\alpha=1$}.
\end{array}\r.
\end{align}
  The result follows from \eqref{eqn:J1low} and \eqref{eqn:J1up}, as $|J_2|<\infty$.
\end{proof}

\section{Proof of Theorem \ref{thm:clt}}
We prove Theorem \ref{thm:clt} in this section. First, we recall the definitions of correlation, truncated correlation functions and  their properties. 
\subsection{Joint intensity functions :}\label{sec:cor} The joint intensity measures of $X=(X_i)_{i\in\Z^d}$. Let $\CC_c(\R^d)$ be the space of compactly supported continuous functions on $\R^d$. 
Let 
\begin{align}\label{eqn:kthintesity}
\rho_k(i_1,\ldots, i_k)=\E[X_{i_1}\cdots X_{i_k}], \mbox{ for $i_1,\ldots, i_k\in \Z^d$}.
\end{align}
Next, for the sake of completeness, we show that $\rho_k$ is the $k$-th intensity function of $X$ with respect to counting measure on $\Z^{kd}$. See \cite[Section 3.3]{dorlas}, \cite{OL} for the details.

\begin{enumerate}
	\item {\it The first intensity function:} Let $\varphi$ be a continuous function on $\R^d$. Define  $\vp(X):=\sum_{i\in \Z^d}\vp(i)X_i$ and  $T_1(\vp)=\E(\vp(X))$. Note that $T_1$ is a continuous linear form on $\mathcal C_c(\R^d)$. Then by the Riesz-Markov (-Kakutani) theorem there is a unique positive Borel measure $\mu_1$  in $\Z^d$ such that 
	\begin{align}\label{eqn:firstintensitymeasure}
	T_1(\vp)=\int \vp(x)d\mu_1(x), \mbox{ for all $\vp\in \CC_c(\R^d)$}.
	\end{align}
	See \cite[p. 10]{manjubook}. The measure $\mu_1$ is called the first intensity measure. Observe that the measure  $\sum_{i\in \Z^d}\rho(i)\delta_i$, where $\delta_x(\cdot)$ denotes the Dirac-delta measure at $x$, satisfies the relation \eqref{eqn:firstintensitymeasure}.
	By the uniqueness of first intensity measure, , $\mu_1=\sum_{i\in \Z^d}\rho(i)\delta_i$. Therefore $\rho(i)$ is the first intensity function of $X$ with respect to the counting measure on $\Z^d$.
	
	\item {\it The second intensity function:} Define a positive bilinear functional on $C_c(\R^d)\times C_c(\R^d)$ by
	\begin{align*}
	T_{2}(\vp,\psi)=\E[\vp(X)\psi(X)],
	\end{align*}
	which induces the a positive linear functional on $C_c(\R^2)$. Then there exists a  unique positive regular Borel measure $\mu_2$ such that 
	\begin{align*}
	T_{2}(\vp,\psi)&=\int_{\R^{2d}}\vp(x)\psi(y)d\mu_2(x,y).
	\end{align*}
	See \cite[p. 11]{manjubook}. Again the measure  $ \sum_{i, j\in \Z^d}\rho_2(i,j)\delta_i\delta_j$ satisfies the last equation. Therefore by uniqueness $\mu_2= \sum_{i, j\in \Z^d}\rho_2(i,j)\delta_i\delta_j$,  and $\rho_2$ is the second intensity function of $X$ with respect to the counting measure on $\Z^{2d}$. In particular if $\rho_2(i,j)=K(i-j)$ with  $\{K(j)\}\in L^1(\Z^d)$ then there exists $S\in L^1([-\pi,\pi]^d)$ such that
	$$
	S(\theta)=\sum_{j\in \Z^d}e^{i\theta.j}K(j), \mbox{ where $\theta\in [-\pi,\pi]^d$}.
	$$
	
	\item {\it The $k$-th intensity function :} Similarly, we have the $k$-th intensity function $\rho_k$ with respect to the counting measure on $\Z^{kd}$, and given by \eqref{eqn:kthintesity}. 
\end{enumerate}
Note that if $X_i\in \{0,1\}$ for $i\in \Lambda$ then $X=(X_i)_{i\in \Lambda}$ is a simple point process in $\Lambda$. If  the $k$-th intensity measure $\mu_k$ is absolutely continuous with respect to the Lebesgue measure, then  we get the joint intensity function in $\Lambda$ in usual sense, as considered in  \cite{MaY}. See \cite[Definition 1.2.2]{manjubook} for more details.

\subsection{Truncated (connected) correlation functions.} \label{sec:truncated}
Correlations between particles are better described by {\it truncated (connected)} correlation functions. These functions are defined recursively, see \cite[Section 3.4]{dorlas}, \cite{OL}, by
\begin{align}\label{eqn:correlation}
&	\rho_1^T(i_1):=\rho_1({i_1}),\nonumber
\\& \rho_n(i_1,\ldots, i_n)=\sum_{\pi\in \mathcal P(n)}\prod_{B\in \pi}\rho_B^T[i_1,\ldots, i_n],
\end{align}
where $\rho_B^T[i_1,\ldots, i_n]=\rho_{|B|}^T(i_j\;:\; j\in B)$, and $\mathcal P(n)$ denotes the set of all partitions of $\{1,\ldots, n\}$, $B$ runs through the list of blocks of the partition $\pi$. The truncated correlation functions can also be written explicitly in terms of the correlation functions as follows 
\begin{align}\label{eqn:truncatedcorrelation}
\rho_n^T(i_1,\ldots, i_n)=\sum_{\pi\in \mathcal P(n)}(|\pi|-1)!(-1)^{|\pi|-1}\prod_{B\in \pi}\rho_B[i_1,\ldots, i_n],
\end{align}
where $|\pi|$ is the number of parts in the partition.  Note that \eqref{eqn:correlation} and \eqref{eqn:truncatedcorrelation} implies that correlation functions are the analogue of the moments and truncated correlation functions are the analogue of the cumulants of a measure.

\subsection{Cumulants.}
Recall that the joint cumulant of  $X_1,\ldots, X_n$ is given by
\begin{align}\label{eqn:cumulant}
\kappa(X_1,\ldots, X_n)&=\sum_{\pi\in \mathcal P(n)}(|\pi|-1)!(-1)^{|\pi|-1}\prod_{B\in \pi}\E\l(\prod_{i\in B}X_i\r),
\end{align}
which can be found in  \url{https://en.wikipedia.org/wiki/Cumulant}.
If some of the random variables are independent of all of the others, then any cumulant involving two (or more) independent random variables is zero. If all $n$ random variables are the same, then the joint cumulant is the $n$-th ordinary cumulant. The joint cumulant holds the multilinearity property, i.e.,
\begin{align}\label{eqn:cumulantlinear}
&\kappa(X_1,\ldots,c_1X_k+c_2X_k',\ldots, X_n)\nonumber
\\&=c_1\kappa(X_1,\ldots,X_k,\ldots, X_n)+c_2\kappa(X_1,\ldots,X_k',\ldots, X_n)
\end{align}
for all $k=1,\ldots, n.$

Now we proceed to prove the theorem.

\begin{proof}[Proof of Theorem \ref{thm:clt}]
	 Marcinkiewicz  showed that the normal distribution is the only distribution whose cumulant generating function is a polynomial, i.e. the only distribution having a finite number of non-zero cumulants.
See Theorem \ref{thm:marc}, \cite[p. 152]{katznelson2004introduction}.	 Let 
	\begin{align*}
	\bar Q_{\BB_L}(X)=\frac{Q_{\BB_L}(X)-\E[Q_{\BB_L}(X)]}{\sqrt{\Var(Q_{\BB_L}(X))}}.
	\end{align*}
It is enough to show that all but finitely many cumulants of $\bar Q_{\BB_L}(X)$ are asymptotically zero.	Recall  $Q_{\BB_L}(X)=\sum_{i\in \BB_L'}X_i$, where $\BB_L'=\{i\in \Z^d\; :\; i\in \BB_L\}$. 
	Equations \eqref{eqn:truncatedcorrelation} and \eqref{eqn:cumulant} imply that
	\begin{align}\label{eqn:tcfandcumulant}
	\rho_n^T(i_1,\ldots, i_n)=\kappa(X_{i_1},\ldots, X_{i_n}),
	\end{align}
	and by the multilinearity  \eqref{eqn:cumulantlinear} we have 
	\begin{align*}
	\kappa_n(Q_{\BB_L}(X))=\sum_{i_1,\ldots,i_n\in \BB_L'}\kappa(X_{i_1},\ldots, X_{i_n})=\sum_{i_1,\ldots,i_n\in \BB_L'}\rho_n^T(i_1,\ldots,i_n).
	\end{align*}
	Since  $\kappa_n(cY)=c^n\kappa_n(Y)$, by  \eqref{eqn:cumulantlinear}, for any random variable $Y$ and $\E[Q_{\BB_L}(X)]=0$, then
	\begin{align*}
	\kappa_n(\bar Q_{\BB_L}(X))=\frac{1}{(\Var(Q_{\BB_L}(X)))^{n/2}}\sum_{i_1,\ldots,i_n\in \BB_L'}\rho_n^T(i_1,\ldots,i_n).
	\end{align*}
Recall that we assume that 
	\begin{align*}
	\sup_{i_1}	\sum_{i_2,\ldots,i_n\in \BB_L'}\rho_n^T(i_1,\ldots,i_n)<\infty.
	\end{align*}
	There exists a positive constant $C$ such that 
	\begin{align*}
	\kappa_n(\bar Q_{\BB_L}(X))\le \frac{C|\BB_L'|}{(\Var(Q_{\BB_L}(X)))^{n/2}}.
	\end{align*}
	By Proposition \ref{lem:lowervariance} we get 
	\begin{align*}
	\kappa_n(\bar Q_{\BB_L}(X))\lesssim \frac{L^d}{L^{(d-1)n/2}}=L^{-\frac{nd}{2}(1-\frac{1}{d}-\frac{2}{n})}.
	\end{align*}
	Note that, for $d>3$ and $n\ge 3$, we have $1-\frac{1}{d}-\frac{2}{n}>0$, and 
	\begin{align*}
	\kappa_n(\bar Q_{\BB_L}(X))\to 0, \mbox{ as $L\to \infty$}.
	\end{align*}
This completes the proof for $d>3$.	

	If $d=2,3$ then we have $1-\frac{1}{d}-\frac{2}{n}>0$ when $n\ge 5$. Therefore, for $n \ge 5$, we have 
	\begin{align*}
	\kappa_n(\bar  Q_{\BB_L}(X))\to 0, \mbox{ as $L\to \infty$}.
	\end{align*}

Therefore we have \[\kappa_1(\bar Q_{\BB_L}(X))=\E[\bar Q_{\BB_L}(X)]=0,  \kappa_1(\bar Q_{\BB_L}(X))=\E[(\bar Q_{\BB_L}(X))^2]=1 \] and \[   \lim_{L\to \infty}\kappa_n(\bar Q_{\BB_L}(X))=0\] for all $n\ge 5$. 
If we already knew that $\bar Q_{\BB_L}(X)$ converges in distribution, then this would demonstrate, via Marcinkiewicz's Theorem, that the limit must be Gaussian, since we would have  only finitely many non-zero cumulants in limit.

Since we do not know a priori that $\bar  Q_{\BB_L}(X)$ converges in distribution, we proceed via a compactness argument as follows. First, we observe that $\mathrm{Var}[\bar  Q_{\BB_L}(X)]=1$, which implies that the random variables $\bar  Q_{\BB_L}(X)$ give rise to a \textit{tight} family of distributions. This implies that, for any sequence $L_i \to \infty$, there exists a sub-sequence $L_{i_j}$ such that $\bar  Q_{\BB_{L_{i_j}}}(X)$ converges to a limiting random variable $\chi$. Using our investigation of the cumulants for the sequence of random variables $\bar  Q_{\BB_{L_{i_j}}}(X)$, we may deduce that such  $\chi$ must be a standard Gaussian. 

Now, let if possible $\bar Q_{\BB_L}(X)$ not converge to a standard Gaussian (for $d=2,3$). Let $\delta$ be a metric that metrizes the topology of distributional convergence in the space of probability measures on $\R^d$  (for instance, the Levy-Prokhorov metric). Let $\mu_L$ denote the probability measure corresponding to the random variable  $\bar Q_{\BB_L}(X)$, and $\mu$ be the distribution of a standard normal.   Then, for some  $\eps>0$, there must be a sequence $L_i \to \infty$  such that $\delta(\mu_{L_i},\mu)>\eps \quad \forall i$. But, using the subsequential argument of the previous paragraph, we may conclude that there is a subsequence $\{L_{i_j}\}_{j \ge 1} \subseteq \{L_i\}_{i\ge 1}$ such that $\bar  Q_{\BB_{L_{i_j}}}(X)$ converges to a standard Gaussian. But this would imply that $\delta(\mu_{L_{i_j}},\mu) \to 0$, whereas we have already noted that we must have  $\delta(\mu_{L_{i_j}},\mu)>\eps \quad \forall j$. This leads us to a contradiction, implying that $\bar Q_{\BB_L}(X)$ must converge to a standard Gaussian as $L \to \infty$ also for $d=2,3$.

 This completes the proof.
\end{proof}

\section{ Proof of Theorem \ref{lem:entropycontinuous}}\label{sec:ent}

In this section we prove Theorem \ref{lem:entropycontinuous} using the  Szeg\" o's theorem. Let $F$ be a real integrable function on $[-\pi,\pi]^d$, and let its Fourier transform is given by
\begin{align*}
	\hat{F}(k)=\frac{1}{(2\pi)^d}\int_{[-\pi,\pi]^d}e^{-ik\cdot \theta}F(\theta)d \theta,
\end{align*}
where $k \cdot\theta=\sum_{j=1}^{d}k_j\theta_j$. Let $T_L(F)=(\hat{F}(i-j))_{|\Lambda_L'|\times |\Lambda_L'|}$. Recall $\Lambda_L=\{Lx\; :\; x\in \Lambda\}$ and $\Lambda_L'=\Z^d\cap \Lambda_L$.
 A multidimensional version of Szeg\" o's theorem is stated below, which will be used in the proof of Theorem \ref{lem:entropycontinuous}.
\begin{theorem}\cite[Theorem 2]{linnik}\label{thm:szego}
Let $F> 0$ be a function on $[-\pi,\pi]^d$ such that 
\begin{align*}
	\sum_{k\in \Z^d}|\hat{F}(k)|, \sum_{k\in \Z^d}|k||\hat{F}(k)|^2<\infty,
\end{align*}  
where $|k|^2=\sum_{j=1}^{d}k_j^2$. Let $\Lambda \subset\R^d$ be as in Theorem \ref{lem:entropycontinuous}. Then as $L\to \infty$
\begin{align*}
	\lim_{L\to \infty}\frac{1}{|\Lambda_L'|}\log\det(T_L(F))=\frac{1}{(2\pi)^d}\int_{[-\pi,\pi]^d}\log F(\theta)d \theta.
\end{align*}
\end{theorem}

%


For the classical Szeg\" o's theorem, we refer to \cite{szego}, \cite{Ibragimov}, \cite{simon} and references there in. The following well known fact will be used in the proof of Theorem \ref{lem:entropycontinuous}.
For the sake of completeness we give a proof.
\begin{lemma}\label{ft:entropy}
	Let $X=(X_1,\ldots, X_d)$ be  a vector of random variables in $\R^d$ with continuous density, and $G=(G_1,\ldots, G_d)$ be a vector of Gaussian random variables in $\R^d$. Suppose $\E[X_i]=\E[G_i]=0$ and $\E[X_iX_j]=\E[G_iG_j]$ for all $i,j\in \{1,\ldots, d\}$. Then
	$$h(X)\le h(G).$$
\end{lemma}

\begin{proof}[Proof of Theorem \ref{lem:entropycontinuous}]
	Lemma \ref{ft:entropy} implies that
	\begin{align}\label{eqn:ent1}
		\frac{h(X|_{\Lambda_L})}{|\Lambda_L'|}\le\frac{h(G|_{\Lambda_L})}{|\Lambda_L'|}, 
	\end{align} 
where  $(G_i)_{i\in \Z^d}$ be the mean zero variance one  Gaussian field with  the given covariance kernel $K$. Now we show that the right hand side diverges to $-\infty $ as $L\to \infty$.

Note that $G|_{\Lambda_L}$ can be thought as a vector of $|\Lambda_L'|$ many Gaussian random variables with mean zero and variance one. Again the joint distribution of Gaussian random variables determined by its kernel. Therefore the joint density of the random variables $G|_{\Lambda_L}$ is given by
	\begin{align*}
	f_L(x_L)=\frac{1}{\sqrt{\det(2\pi \Sigma_L)}}e^{-\frac{1}{2}x_L^t\Sigma_L^{-1}x_L},
	\end{align*}
	where $x_L$ is a vector of length $|\Lambda_L'|$ and $\Sigma_L$ is the covariance kernel matrix for the random variables $\{G_i \; :\; i\in \Lambda_L'\}$. In other words $\Sigma_L=(K(i-j))_{|\Lambda_L'|\times |\Lambda_L'|}$.
		Then
	\begin{align*}
	\log f_L(x_L)=-\frac{|\Lambda_L'|\log (2\pi)}{2}-\frac{\log \det(\Sigma_L)}{2}-\frac{1}{2}x_L^t\Sigma_L^{-1}x_L.
	\end{align*}
	Therefore the entropy of $G|_{\Lambda_L}$ is given by
	\begin{align}\label{eqn:ent2}
&	h(G|_{\Lambda_L})=-\int_{\R^{|\Lambda_L'|}} f_L(x)\log f_L(x)dx\nonumber
	\\=&\frac{|\Lambda_L'|\log (2\pi)}{2}+\frac{\log \det(\Sigma_L)}{2}+\frac{1}{\sqrt{\det(2\pi \Sigma_L)}}\int (\frac{1}{2}x^t\Sigma_L^{-1}x) e^{-\frac{1}{2}x^t\Sigma_L^{-1}x}dx.
	\end{align}
	By the change of variables formula, putting $y=\Sigma^{1/2}x$, we get
	\begin{align}\label{eqn:ent3}
		\frac{1}{\sqrt{\det(2\pi \Sigma_L)}}\int_{\R^{|\Lambda_L'|}} (x^t\Sigma_L^{-1}x) e^{-x^t\Sigma_L^{-1}x}dx=&\frac{1}{2}\frac{1}{\sqrt{(2\pi )^{|\Lambda_L'|}}}\int_{\R^{|\Lambda_L'|}} (y^ty) e^{-\frac{1}{2}y^ty}dy\nonumber
		\\=&\frac{1}{2}\sum_{i\in \Lambda_L'}\frac{1}{\sqrt{2\pi}}\int_{\R}y_i^2e^{-\frac{y_i^2}{2}}dy_i=\frac{|\Lambda_L'|}{2}.
	\end{align}
Therefore using \eqref{eqn:ent3} from \eqref{eqn:ent2} we get
\begin{align*}
		h(G|_{\Lambda_L})=\frac{|\Lambda_L'|\log (2\pi)}{2}+\frac{\log \det(\Sigma_L)}{2}+\frac{|\Lambda_L'|}{2}.
\end{align*}
Which implies that
	\begin{align}\label{eqn:reduction}
	\lim_{L\to \infty}\frac{h(G|_{\Lambda_L})}{|\Lambda_L'|}=\frac{1}{2}+\frac{\log (2\pi)}{2}+\lim_{L\to \infty} \frac{\log \det(\Sigma_L)}{2|\Lambda_L'|}.
	\end{align}
	
	Next we calculate $\det(\Sigma_L)$ using the strong Szeg\" o's theorem.  Note that we have  $\Sigma_L=(K(i-j))_{\Lambda_L'\times \Lambda_L'}$. Recall \eqref{eqn:S}, we have
	\begin{align*}
		S(\theta)=\sum_{j\in \Z^d}K(j)e^{ij\cdot\theta}, \mbox{ where  $\theta \in [-\pi,\pi]^d$}
	\end{align*} 
	and $S\ge 0$ on $[-\pi,\pi]^d$.  Theorem \ref{thm:szego} can not be applied directly for $\Sigma_L$ as $S$ vanishes near the origin. We perturb the structure function to apply the Szeg\" o's theorem. Let $\epsilon>0$ and $S_{\epsilon}(\theta)=S(\theta)+\epsilon$ be a modified structure function. Its Fourier coefficient are given by
	\begin{align*}
		K_{\epsilon}(j)=\frac{1}{(2\pi)^d}\int_{[-\pi,\pi]^d} e^{-ij\cdot \theta}S_{\epsilon}(\theta)d\theta=K(j)+\epsilon \delta_0(j).
	\end{align*}
	Let $\Sigma_L^{\epsilon}=(K_{\epsilon}(i-j))_{|\Lambda_L'|\times |\Lambda_L'|}$. Then $\Sigma_L^{\eps}=\Sigma_L+\epsilon I$. Therefore we have
	\begin{align*}
		\det (\Sigma_L)\le \det (\Sigma_L^{\eps}),
	\end{align*}
	as $\Sigma_L$ is non-negative definite.
	Therefore  using Theorem \ref{thm:szego} we get
	\begin{align}\label{eqn:szegoap}
	\limsup_{L\to \infty} \frac{\log \det(\Sigma_L)}{|\Lambda_L'|}\le \limsup_{L\to \infty}\frac{\log \det(\Sigma_L^{\eps})}{|\Lambda_L'|}=\frac{1}{(2\pi)^d}\int_{[-\pi,\pi]^d}\log S_{\eps}(\theta)d \theta.
	\end{align}
	The right hand side of \eqref{eqn:szegoap} goes to $-\infty$ as $\eps\to 0$, since $S$ is bounded and zero in a neighbourhood  of the  origin. Therefore from  \eqref{eqn:szegoap} we get
	\begin{align}\label{eqn:ent4}
	\lim_{L\to \infty}\frac{h(G|_{\Lambda_L})}{|\Lambda_L'|}=-\infty.
	\end{align}
	We conclude the result from  \eqref{eqn:ent1} and \eqref{eqn:ent4}.

 We note in the passing that, even if the structure function $S$ does not vanish near the origin, as soon as $S$ fails to be logarithmically integrable, sending $\eps \to 0$ in \eqref{eqn:szegoap} we may deduce that the asymptotic entropy per site is still $-\infty$. Thus, entropic degeneracy already sets in under milder conditions than actual vanishing of the structure function in a neighbourhood of the origin.
\end{proof}

\section{Appendix}

For the sake of completeness, the Riemann-Lebesgue lemma is stated below, see  \cite[Lemma 4.2.1]{epstein}.
\begin{lemma}[Riemann Lebesgue lemma]\label{ft:RL}
	If $f$ is $L_1$ integrable on $\R^d$, that is to say, if the Lebesgue integral of $|f|$ is finite, then the Fourier transform of $f$ satisfies
	$$
	\hat{f}(z):=\int_{\mathbb{R}^d} f(x) \exp(-iz \cdot x)\,dx \rightarrow 0\text{ as } |z|\rightarrow \infty.
	$$
\end{lemma}

\begin{theorem}\cite[P. 152]{katznelson2004introduction}\label{thm:marc}
	If $e^{P(\xi)}$ is the Fourier-Stieltjes transform of a positive measure, with $P$ a polynomial, then $\deg P\le 2$.
\end{theorem}

For the sake of completeness, the proof of Lemma \ref{ft:entropy} is given below.
\begin{proof}[Proof of Lemma \ref{ft:entropy}]
	Let $g$ be the joint density function of the random  vector $G$. Then  $g$ is given by, set $x=(x_1,\ldots, x_d)$,
	\begin{align*}
		g(x)=\frac{1}{\det(\sqrt{2\pi\Sigma_d})}e^{-\frac{1}{2}x^t\Sigma_d^{-1} x},
	\end{align*}
where $\Sigma_d=(\sigma(i,j))_{d\times d}$ with $\sigma(i,j)=\E[X_iX_j]$.
	Let $f$ be the continuous density function of the  random variable $X$.  The relative entropy (also known as Kullback-Leibler divergence) between $f$ and $g$ is given by
	\begin{align*}
		0\le D_{\emph{KL}}(f\|g)=\int f(x)\log (\frac{f(x)}{g(x)})dx=-h(X)-\int f(x)\log (g(x))dx.
	\end{align*}
	Note  that we have 
	$$
	\log (g(x))=-\frac{1}{2}\log \det(2\pi\Sigma_d)-\frac{1}{2}x^t\Sigma_d^{-1} x.
	$$
	Which implies that, as $\E[X_iX_j]=\E[G_iG_j]$ for all $i,j$, 
	\begin{align*}
		\int_{-\infty}^\infty f(x)\log (g(x))dx&=-\frac{1}{2}\log \det(2\pi\Sigma_d)-\frac{1}{2}\int f(x)(x^t\Sigma_d^{-1} x)dx
		\\&=-\frac{1}{2}\log \det(2\pi\Sigma_d)-\frac{1}{2}\int g(x)(x^t\Sigma_d^{-1} x)dx=-h(G).
	\end{align*}
	Thus we have
	$
	h(G)-h(X)\ge 0.
	$
	  Moreover, the properties of Kullback-Leibler divergence imply that $h(X)=h(G)$ when  $f=g$.  Hence the result.
\end{proof}

\noindent{\bf Acknowledgements:} The research of Kartick Adhikari is supported by Zeff Fellowship, Viterbi Fellowship and Israel Science Foundation, Grant 2539/17 and Grant 771/17.  The work of Subhroshekhar Ghosh is supported by MOE grant R-146-000-250-133. The work of Joel L. Lebowitz is supported by AFOSR Grant FA9550-16-1-0037.

We thank the referees for insightful comments and suggestions.

\bibliography{bibfluctuations}
\bibliographystyle{alpha}

\end{document}

%% file: mypreamble.tex
\theoremstyle{theorem}
    \newtheorem{theorem}{Theorem}
    \newtheorem{lemma}{Lemma}
    \newtheorem{proposition}{Proposition}
    \newtheorem{corollary}[theorem]{Corollary}

\theoremstyle{definition} 

    \newtheorem{remark}{Remark}
    \newtheorem{example}[theorem]{Example}
    \newtheorem{exercise}[theorem]{Exercise}

\def\BB{{\mathcal B}}
\def\CC{{\mathcal C}}

\def\vp{\varphi}

\def \CC{{\mathcal C}}

\def\Z{\mathbb{Z}}

\def\N{\mathbb{N}}

\def\R{\mathbb{R}}

\def\S{\mathbb{S}}

\def\Z{\mathbb{Z}}

\def\n{{\bf n}}

\def\l{\left}
\def\r{\right}
\def\<{\langle}
\def\>{\rangle}

\newcommand{\one}{{\mathbf 1}}
\newcommand{\E}{\mbox{\bf E}}

\def\bar{\overline}

\def\Var{\mbox{Var}}
\def\cov{\mbox{Cov}}

\newcommand\mnote[1]{} 
\newcommand\be{\begin{equation*}}

\newcommand\ee{\end{equation*}}

\newcommand\ben{\begin{equation}}
\newcommand\een{\end{equation}}
\newcommand\bes{\begin{eqnarray*}}
\newcommand\ees{\end{eqnarray*}}

\newcommand\bex{\begin{exercise}}
\newcommand\eex{\end{exercise}}
\newcommand\beg{\begin{example}}
\newcommand\eeg{\end{example}}
\newcommand\benu{\begin{enumerate}}
\newcommand\eenu{\end{enumerate}}
\newcommand\beit{\begin{itemize}}
\newcommand\eeit{\end{itemize}}
\newcommand\berk{\begin{remark}}
\newcommand\eerk{\end{remark}}
\newcommand\bdefn{\begin{defintion}}
\newcommand\edefn{\end{definition}}
\newcommand\bthm{\begin{theorem}}
\newcommand\ethm{\end{theorem}}
\newcommand\bprf{\begin{proof}}
\newcommand\eprf{\end{proof}}
\newcommand\blem{\begin{lemma}}
\newcommand\elem{\end{lemma}}

\newcommand{\sm}{{\raise0.3ex\hbox{$\scriptstyle \setminus$}}}

\def\l{\left}
\def\r{\right}

\def\eps{\epsilon}

\newcommand{\imply}{\;\;\;\Longrightarrow\;\;\;}

\def\CHI{\mathchoice%
{\raise2pt\hbox{$\chi$}}%
{\raise2pt\hbox{$\chi$}}%
{\raise1.3pt\hbox{$\scriptstyle\chi$}}%
{\raise0.8pt\hbox{$\scriptscriptstyle\chi$}}}
\def\smalloplus{\raise1pt\hbox{$\,\scriptstyle \oplus\;$}}